%% file: main.tex
\documentclass[11pt,a4paper]{amsart}
\usepackage[utf8]{inputenc}
\usepackage[T1]{fontenc}

\pdfoutput=1

\title[Separation profiles of  hyperbolic apex-minor-free graphs]{Separation profiles of hyperbolic \\ planar and apex-minor-free graphs}
\date{8th July 2026}

\author{Joseph P. MacManus}
\address{School of Mathematics,  University of Bristol, Bristol, BS8 1UG, UK, and the Heilbronn Institute for Mathematical Research, Bristol, UK}
\email{joseph.macmanus@bristol.ac.uk}

\author{Bobby Miraftab}
\address{School of Computer Science, 5302 Herzberg Building, 1125 Colonel By Drive, Ottawa, ON, K1S 5B6, Canada}
\email{bobby.miraftab@gmail.com}

\usepackage{amsfonts}
\usepackage{amsmath}
\usepackage{amsthm}
\usepackage{amssymb}
\usepackage{stmaryrd}
\usepackage{mathrsfs}

\usepackage{graphicx}
\usepackage[dvipsnames]{xcolor}
\usepackage{tikz}
\usepackage{quiver}
\usepackage[colorlinks]{hyperref} % should usually be imported last
\hypersetup{
    linkcolor=black,
    citecolor=blue,
}
\usepackage{fancyref}
\usepackage{todonotes}
\usepackage{mathabx}

\usepackage{thmtools, thm-restate}
\usepackage{cleveref} 
\usepackage{mathtools}
\usepackage[abbrev]{amsrefs}

\input{commands}

\begin{document}

\begin{abstract}
    We show that the separation profile of a hyperbolic planar graph and, more generally, a hyperbolic apex-minor-free graph, grows at most logarithmically, answering a question of Benjamini, Schramm, and Tim\'ar in the affirmative. 
    % This generalises the known fact that hyperbolic planar graphs have logarithmic tree-width, to subgraphs of such graphs. 
    % Our proof involves showing that finite subgraphs of hyperbolic planar graphs have logarithmic tree-width, by constructing geodesic loaded cycles within them and applying Gromov's tree approximation lemma.
\end{abstract}

\maketitle

% \newpage

% \tableofcontents

% \newpage

\section{Introduction}
% \todo[inline]{Tree-decomposition and tree-width}

The \emph{separation profile} of an infinite graph was introduced by Benjamini, Schramm, and Tim\'ar in \cite{benjamini2012separation}. For bounded-degree graphs, the asymptotic behaviour of the separation profile is monotone under regular maps (e.g. coarse and quasi-isometric embeddings), making it an invaluable tool in coarse geometry for obstructing the existence of such maps. Thus, the computation of separation profiles and related invariants
% (and, more generally, \emph{Poincar\'e profiles}) 
has been a topic of interest within geometric group theory and coarse geometry; see, for example, \cites{hume2022poincare, hume2020poincare, benjamini2012separation, gournay2023separation}. The  definition runs as follows.

\begin{definition}[Separation profile]
    Let $\Gamma$ be a graph. Let $H \subset \Gamma$ be a subgraph. Define $\cut(H)$ as the minimal number of vertices one can remove from $H$ so that every connected component of the complement contains at most $\frac12|V(H)|$ vertices. 
Then, define the \emph{separation profile} of $\Gamma$ via 
$$
\sep_\Gamma(n) = \max \{\cut(H) : H \subset \Gamma ,  \ |V(H)| \leq n\}.
$$
\end{definition}

In words, the separation profile measures how hard finite subgraphs are to separate. It follows from the famous planar separator theorem of Lipton--Tarjan  that $\sep_\Gamma(n) = \cO(\sqrt n)$ for any planar graph $\Gamma$ \cite{lipton1979separator}. Moreover, it was shown in \cite{benjamini2012separation} that (any bounded-degree graph quasi-isometric to) the hyperbolic plane has a separation profile which grows logarithmically. The authors posited the question of whether this observation extends to all planar graphs which are hyperbolic in the sense of Gromov {\cite[Question~4.5]{benjamini2012separation}}.
In this paper we answer their question in the affirmative, and actually prove something more general.

\begin{restatable}{alphtheorem}{mainapex}\label{thm:main-apex}
For every $\delta \geq 0$ and every finite apex graph $A$, there exists a constant $C = C(A,\delta) > 0$ such that
every connected, $A$-minor-free, $\delta$-hyperbolic graph $\Gamma$ satisfies
\[
    \sep_\Gamma(n)\le C\log_2(n+1)
\]
for every $n\ge 1$.
\end{restatable}

Recall that a graph $A$ is an \emph{apex} graph if
$A-a$ is planar for some vertex $a\in V(A)$. Since $K_5$ is apex and every planar graph is $K_5$-minor-free, this theorem
applies to hyperbolic planar graphs, but also much more broadly. For example, if a graph $\Gamma$ embeds into some finite-genus surface, then it cannot contain the apex graph $K_{3,t}$ as a minor for all sufficiently large $t > 0$.

In the particular case of planar graphs, we also present a different proof which yields much tighter quantitative control on the constants involved, with explicit linear dependence on $\delta$. 

\begin{restatable}{alphtheorem}{mainplanar}\label{thm:main-planar}
    Every $\delta$-hyperbolic planar graph $\Gamma$ satisfies
    $$
    \sep_\Gamma(n) \leq   60 + 120\delta \log_2(n)
    $$
    for all $n \geq 1$.
\end{restatable}

We make no claim that these constants are tight, and there is likely to be room for optimisation.

Our results can also be interpreted as statements about tree-width.
It is a known fact that any finite $\delta$-hyperbolic planar graph $\Gamma$ on $n$ vertices satisfies $\tw(\Gamma) = \cO(\delta \log (n)+1)$. This follows from results of Chepoi \textit{et al.} \cite{chepoi2008diameters} and Dieng--Gavoille \cite{dieng2009tree}; see also a recent paper of Kisfaludi-Bak \textit{et al.} on separators in hyperbolic planar graphs \cite{kisfaludi2023separator}. Our results imply an extension of this observation, applying more generally to subgraphs of hyperbolic apex-minor-free graphs. 

\begin{alphcor}\label{cor:tw}
    For every $\delta \geq 0$ and every finite apex graph $A$, there exists a constant $D = D(A,\delta) > 0$ such that the following holds.
    Let $H$ be a finite graph, and suppose $H$ embeds as a subgraph into a connected, $A$-minor-free, $\delta$-hyperbolic graph. Then
    \[
        \tw(H)\le D\log_2(|V(H)|+1).
    \]
\end{alphcor}

It is important to note that the graph $H$ in the above corollary may itself be very far from being $\delta$-hyperbolic. Corollary~\ref{cor:tw} essentially follows from Theorem~\ref{thm:main-apex} together with an application of a theorem of Dvo{\v{r}}{\'a}k--Norin \cite{dvovrak2019treewidth}; see the recent note of Hume \cite{hume2026separation} for details.
If $H$ includes as a subgraph into a $\delta$-hyperbolic \textbf{planar} graph, then we may take $D = \cO(\delta+1)$ by instead applying the explicit bound of Theorem~\ref{thm:main-planar} in place of Theorem~\ref{thm:main-apex}. 

We would also like to mention a related result of Berger, Kenyon, Mossel, and Peres \cite{berger2005glauber}, who proved that the \textit{cut-width} of balls in bounded-degree hyperbolic planar graphs is at most logarithmic in the number of vertices. It is natural to define the \emph{cut-width profile} of a graph, by simply replacing the quantity $\cut(H)$ in the definition of the separation profile with the cut-width of $H$; in simple terms, we wish to efficiently separate $H$ by removing edges rather than vertices. The corresponding question of whether the cut-width profile of a hyperbolic planar graph is necessarily logarithmic remains open. 

\subsection{Outline}

We now discuss the proofs of Theorems~\ref{thm:main-apex} and \ref{thm:main-planar}. Both results make use of Gromov's tree approximation lemma \cites{gromov1987hyperbolic,ghys2013groupes} to construct tree-decompositions. Indeed, it is well-known that tree approximation can be used to place a logarithmic upper bound on the tree-length of a finite hyperbolic graph.  Of course, we are interested in \emph{subgraphs} of such graphs, and as mentioned above, a subgraph of a $\delta$-hyperbolic graph may be far from being $\delta$-hyperbolic itself. Our two approaches deal with this failure in two distinct ways.

In the proof of Theorem~\ref{thm:main-apex}, we overcome this failure of hyperbolicity by `correcting' our finite subgraphs. More precisely, given a $\delta$-hyperbolic graph $\Gamma$ and a finite connected subgraph $H \subset \Gamma$ on $n\geq 1$ vertices, we construct an efficient intermediate `hull' 
$$
H \subset K \subset \Gamma,
$$
such that $K$ is uniformly quasi-isometrically embedded in $\Gamma$. By `efficient', we mean that 
$$
|V(K)| = \cO(n^6).
$$
This construction may be of independent utility; see Proposition~\ref{lem:hull} below. Importantly, the fact that $K$ is quasi-isometrically embedded implies that it is $\delta'$-hyperbolic for some $\delta' \geq 0$ depending only on $\delta$. Gromov's tree approximation then gives
logarithmic tree-length for $K$ as mentioned earlier.  A theorem of
Coudert--Ducoffe--Nisse \cite{CoudertDucoffeNisse2016} bounds tree-width in terms of tree-length in
apex-minor-free graph classes, and a weighted tree-decomposition separator 
finally gives the desired logarithmic bound for the separation profile.

Proving Theorem~\ref{thm:main-planar}, we take a different approach. As mentioned above, it was recently observed by Hume \cite{hume2026separation} that, as a consequence of a theorem of Dvo{\v{r}}{\'a}k--Norin \cite{dvovrak2019treewidth} (see also \cite{houdrouge2025separation}), to understand the growth of the separation profile it is sufficient to bound the tree-width of its subgraphs.
Hence, we are able to prove Theorem~\ref{thm:main-planar} by bounding the tree-width of subgraphs of hyperbolic planar graphs. Since the tree-width of a planar graph corresponds closely to the size of its largest grid minor \cites{grigoriev2011tree,seymour1994call}, our problem further reduces to studying the size of grid minors in hyperbolic planar graphs. This leads us to instead work with the \emph{grid profile} of a planar graph, which is asymptotically equivalent to the separation profile (for planar graphs); see  \S\ref{sec:gridprofile} below.
In order to effectively bound the grid profile logarithmically, we apply the tree approximation lemma  directly to our subgraph. This yields a tree-decomposition $(T, \beta)$ of $H$ with strong control with respect to the ambient metric (though generally terrible with respect to the intrinsic metric of $H$). More precisely, this decomposition satisfies
$$
\diam_\Gamma(\beta(v)) \leq 1 + 2\delta \log_2(|V(H)|),
$$
for every $v \in V(T)$; see Proposition~\ref{prop:tree-decomp-from-tree-approx} below. 
From here, we apply a tool of Berger--Seymour \cite{berger2024bounded}, and use the Jordan curve theorem to construct \emph{geodesic loaded cycles} within our grid minors to obtain lower bounds for the ambient diameter of certain bags, which yields our explicit logarithmic upper bound with linear dependence on $\delta$.

\subsection*{Acknowledgements}

The first named author was supported by the Additional Funding Programme for Mathematical Sciences, delivered by EPSRC (EP/V521917/1) and the Heilbronn Institute for Mathematical Research. 
He also thanks David Hume and John Mackay for comments. 
The second author was supported by the Natural Sciences and Engineering Research Council of Canada (NSERC).

% \newpage

\section{Preliminaries}

We first set out conventions and basic definitions. 

\subsection{Graphs}

Throughout this paper, all graphs are, unless explicitly stated otherwise, taken to be simple, undirected, and unweighted. Given a graph $\Gamma$, we denote by $V(\Gamma)$ its vertex set and $E(\Gamma)$ its (unoriented) edges. We may sometimes denote an edge with endpoints $x$, $y$ by $xy$. Metrise $V(\Gamma)$ with its path metric. Given a subset $U \subset V(\Gamma)$, we denote by $N_\Gamma[U]$ its \emph{neighbourhood}. That is, the set of vertices which lie at distance at most 1 from $U$. If $U = \{v\}$ we may simply write $N_\Gamma[v]$. If $H$ is a minor of $\Gamma$, we may write $H \prec \Gamma$. 

Given a graph $\Gamma$, a \emph{tree-decomposition} of $\Gamma$ is an ordered pair $(T, \beta)$ where $T$ is a simplicial tree, and $\beta : V(T) \to 2^{V(\Gamma)}$ is a map satisfying the following axioms:
\begin{enumerate}
    \item[(T1)] We have $V(\Gamma) = \bigcup_{v \in V(T)} \beta(v)$.

    \item[(T2)] For all $xy \in E(\Gamma)$, there exists $v \in V(T)$ such that $x,y \in \beta(v)$.

    \item[(T3)] Given $x \in V(\Gamma)$, the \emph{fibre} of $x$, $\beta^{-1}(x) := \{v \in V(T): x \in \beta(v)\}$, is connected. 
\end{enumerate}
The subsets $\beta(v)$ are called the \emph{bags} of the tree-decomposition. The \emph{width} of $(T,\beta)$ is defined to be the quantity
$\width(T,\beta) = \sup\{|\beta(v)| : v \in V(T)\} - 1$. The \emph{tree-width} of $\Gamma$, denoted $\tw(\Gamma)$, is 
$$
\tw(\Gamma) := \inf \width(T,\beta),
$$
where the infimum is taken over all tree-decompositions of $\Gamma$. Similarly, the (outer)-\emph{length} of $(T,\beta)$ is the quantity
$$
    \length(T, \beta) = \sup\{\diam_\Gamma(\beta(v)): {v \in V(T)}\}.
$$
The \emph{tree-length} of $\Gamma$, denoted $\tl(\Gamma)$, is then defined to be the infimal length over all tree-decompositions of $\Gamma$.

% \subsection{Coarse geometry}

\subsection{Hyperbolicity}

Given $\delta \geq 0$, we adopt the convention that a graph (or more generally, a geodesic space) $\Gamma$ is \emph{$\delta$-hyperbolic} if it satisfies the \emph{four point condition} with constant $\delta$. That is, 
$$
\dist(x,z) + \dist(y,w) \leq \max\{\dist(x,y) + \dist(z,w), \dist(x,w) + \dist(z,y)\} + 2\delta
$$
for all $w,x,y,z \in V(\Gamma)$. 
% There are, of course, several other equivalent definitions, and choosing a different convention has the effect of replacing $\delta$ by some constant multiple; see, for example, \cites{ghys2013groupes,chepoi2012additive} for discussions.

An important property of hyperbolic graphs is that their triangles are \emph{slim}. That is, given $\rho \geq 0$, a geodesic triangle 
% with sides $P_{xy},P_{yz},P_{zx}$ 
is called \emph{$\rho$-slim} if each side is contained in the (closed) $\rho$-neighbourhood of the
union of the other two sides. 
% i.e.
% \[
% V(P_{xy})\subseteq N_\rho\bigl(V(P_{yz})\cup V(P_{zx})\bigr),
% \]
% and similarly for the other two sides.
Similarly, 
a geodesic quadrilateral 
% with sides $P_1,P_2,P_3,P_4$ 
is \emph{$\rho$-slim} if
each side is contained in the $\rho$-neighbourhood of the union of the other
three sides. With that defined, we briefly record the following well-known fact for later use. 

\begin{lemma} {\rm \cite[\S2, Prop.~1]{ChepoiDraganEstellonHabibVaxesXiang2012}}
\label{lem:thin}
If $G$ is a $\delta$-hyperbolic graph, then every geodesic triangle in $G$
is $4\delta$-slim. Consequently, every geodesic quadrilateral in $G$ is
$8\delta$-slim.
\end{lemma}

% \[
% N_\rho(A)\coloneqq \{v\in V(G): \dist_\Gamma(v,A)\le \rho\}.
% \]

% \newpage

%%%%%%%%%%%%%%%%%%%%%%%%%%%%%%%%%%%%%%%%%%%%%%%%%%%%%%%%%%%%%%%%%%%%%%%%%%%%%%%%%%%%%%%%%%%%
%%%%%%%%%%%%%%%%%%%%%%%%%%%%%%%%%%%%%%%%%%%%%%%%%%%%%%%%%%%%%%%%%%%%%%%%%%%%%%%%%%%%%%%%%%%%
%%%%%%%%%%%%%%%%%%%%%%%%%%%%%%%%%%%%%%%%%%%%%%%%%%%%%%%%%%%%%%%%%%%%%%%%%%%%%%%%%%%%%%%%%%%%

\section{Tree-decompositions from tree approximation}

In order to make use of hyperbolicity, the plan is to use tree approximation in order to construct controlled tree-decompositions of (finite subgraphs of) hyperbolic graphs. We will make use of Gromov's tree approximation lemma \cite[pp.~155-157]{gromov1987hyperbolic}; see also \cite[\S2, Thm.~12]{ghys2013groupes}. The precise statement we give below is taken from {\cite[\S2, Thm.~1]{ChepoiDraganEstellonHabibVaxesXiang2012}}. 

\begin{theorem}[Gromov's tree approximation lemma]\label{thm:tree-approx}
    Let $\Gamma$ be a connected, $\delta$-hyperbolic graph. Let $U  \subset V(\Gamma)$ be a collection of $n$ distinct vertices. Then there exists an  $\R$-tree $\mathcal T$ and a map $f : U \to \mathcal T$ such that 
    $$
    \dist_\Gamma(u, v) - 2\delta \log_2(n) \leq \dist_{\mathcal T}(f(u),f(v)) \leq \dist_\Gamma(u,v)
    $$
    for all $u, v \in U$. 
\end{theorem}

We will apply the above to prove the following proposition.

\begin{proposition}\label{prop:tree-decomp-from-tree-approx}
    Let $\Gamma$ be a connected, $\delta$-hyperbolic graph. Let $H \subset \Gamma$ be a finite subgraph. Then there exists a tree-decomposition $(T,\beta)$ of $H$ such that 
    $$
    \diam_\Gamma(\beta(v)) \leq 1 + 2\delta \log_2(|V(H)|),
    $$
    for all $v \in V(T)$.
\end{proposition}

\begin{proof}
To ease notation, set $n \coloneqq |V(H)|$ and 
$\alpha\coloneqq 2\delta\log_2 n$.
Applying Theorem~\ref{thm:tree-approx} to $\Gamma$ with \(U=V(H)\), we obtain an $\R$-tree $\mathcal T$ and a map $f\colon V(H)\to \mathcal T$
such that, for all \(x,y\in V(H)\),
\[
\dist_\Gamma(x,y)-\alpha\le \dist_{\mathcal T}(f(x),f(y))\le \dist_\Gamma(x,y).
\tag{1}
\]
For every edge \(xy\in E(H)\), we have \(\dist_\Gamma(x,y)=1\), and hence by \((1)\) we have that
$$
\dist_{\mathcal T}(f(x),f(y))\le 1.
$$
Let \(m_{xy}\) be the midpoint of the segment \([f(x),f(y)]\) in \(\mathcal T\). 
Then
\[
\dist_{\mathcal T}(f(x),m_{xy})\le \frac12,
\qquad
\dist_{\mathcal T}(f(y),m_{xy})\le \frac12 .
\tag{2}
\]
Set 
$
S\coloneqq \{f(v):v\in V(H)\}\cup \{m_{xy}:xy\in E(H)\}
$. 
Let \(T'\) be the finite simplicial tree obtained from the convex hull of \(S\) in \(\mathcal T\), by taking as vertices all points of \(S\) together with all branch points of this convex hull. For \(p\in V(T')\), define
\[
\beta(p)\coloneqq \{v\in V(H): \dist_{\mathcal T}(f(v),p)\le \tfrac12\}.
\]
We claim that \((T',\beta)\) is a tree-decomposition of \(H\). First, every vertex \(v\in V(H)\) lies in some bag, since \(f(v)\in V(T')\) and $\dist_{\mathcal T}(f(v),f(v))=0\le \frac12$,
so \(v\in \beta(f(v))\).
Second, every edge is covered. Indeed, if \(xy\in E(H)\), then \(m_{xy}\in V(T')\), and by \((2)\), we get
$x,y\in \beta(m_{xy})$.
Third, the bags containing a fixed vertex form a connected subtree of \(T'\). Fix \(v\in V(H)\). Then
\[
\{p\in V(T'):v\in \beta(p)\}
=
\{p\in V(T'): \dist_{\mathcal T}(f(v),p)\le \tfrac12\}.
\]
This is the set of vertices of \(T'\) lying in the closed ball of radius \(1/2\) around \(f(v)\) in \(\mathcal T\). Closed balls in an \(\mathbb R\)-tree are convex, and so if \(p,q\) belong to this set, then every vertex of \(T'\) on the \(p\)-to-\(q\) path also belongs to it. Therefore this set is connected.
Thus \((T',\beta)\) is a tree-decomposition of \(H\).

It remains to bound the ambient diameter of each bag as a subset of $\Gamma$. Let \(p\in V(T')\), and let \(x,y\in \beta(p)\). 
Then one can see that
\[
\dist_{\mathcal T}(f(x),p)\le \frac12,
\qquad
\dist_{\mathcal T}(f(y),p)\le \frac12.
\]
Therefore
\[
\dist_{\mathcal T}(f(x),f(y))
\le \dist_{\mathcal T}(f(x),p)+\dist_{\mathcal T}(p,f(y))
\le 1.
\]
Using the left inequality in \((1)\), we get
\[
\dist_\Gamma(x,y)-\alpha
\le \dist_{\mathcal T}(f(x),f(y))
\le 1.
\]
Hence we obtain $\dist_\Gamma(x,y)\le 1+\alpha
=1+2\delta\log_2 n$.
% Since \(\dist_\Gamma(x,y)\) is an integer, we have $\dist_\Gamma(x,y)\le \left\lfloor 1+2\delta\log_2 n\right\rfloor$.
Thus every bag has diameter in $\Gamma$ at most
$ 1+2\delta\log_2 n$.
\end{proof}

Applying Proposition~\ref{prop:tree-decomp-from-tree-approx} to the specific case where $H = \Gamma$, we recover the following.

\begin{corollary}\label{lem:tl-from-gromov}
If \(\Gamma\) is a finite, connected, \(\delta\)-hyperbolic graph, then 
$$
\tl(\Gamma)\le  1+2\delta\log_2 (|V(\Gamma)|).
$$
\end{corollary}

In the following sections, the above statements will play an important role in both of the proofs we present.

% \newpage

%%%%%%%%%%%%%%%%%%%%%%%%%%%%%%%%%%%%%%%%%%%%%%%%%%%%%%%%%%%%%%%%%%%%%%%%%%%%%%%%%%%%%%%%%%%%
%%%%%%%%%%%%%%%%%%%%%%%%%%%%%%%%%%%%%%%%%%%%%%%%%%%%%%%%%%%%%%%%%%%%%%%%%%%%%%%%%%%%%%%%%%%%
%%%%%%%%%%%%%%%%%%%%%%%%%%%%%%%%%%%%%%%%%%%%%%%%%%%%%%%%%%%%%%%%%%%%%%%%%%%%%%%%%%%%%%%%%%%%

% \newpage

\section{Separation profiles via efficient hulls}

In this section we walk through the proof of Theorem~\ref{thm:main-apex}.

\subsection{Efficient hulls}

The first step is, given a connected subgraph $H$ of our $\delta$-hyperbolic graph $\Gamma$, to construct an `efficient hull' of this subgraph which is $\delta'$-hyperbolic for some uniform $\delta' \geq 0$ depending only on $\delta$. 
We shall use the following standard stability property of hyperbolicity under quasi-isometric embeddings.

\begin{lemma}
{\rm\cite[\S\hspace{0pt}III.H, Thm.~1.9]{BridsonHaefliger1999}}
\label{lem:qi-hyp}
For every $\Delta,A,\varepsilon$ there exists
$\Delta'=\Delta'(\Delta,A,\varepsilon)$ such that the following holds.
If $X$ and $Y$ are geodesic metric spaces, $Y$ is $\Delta$-hyperbolic, and
$f:X\to Y$ is an $(A,\varepsilon)$-quasi-isometric embedding, then $X$ is
$\Delta'$-hyperbolic.
\end{lemma}

More precisely, we prove the following proposition, which may be of independent interest.

\begin{proposition}\label{lem:hull}
For all $\delta \geq 0$ there exists $\Delta_\delta \geq 0$ such that the following holds. 
Let $\Gamma$ be a connected $\delta$-hyperbolic graph and let
$H\subseteq \Gamma$ be a finite connected subgraph with $|V(H)|=n\ge 2$.
Then there is a finite connected subgraph $F\subseteq \Gamma$
 such that:
\begin{enumerate}
    \item $H\subseteq F$,
    \item $|V(F)|\le (16\delta+3)n^6$,
    \item $F$ is $\Delta_\delta$-hyperbolic.
\end{enumerate}
\end{proposition}

\begin{proof}
For every unordered pair of distinct vertices $x,y\in V(H)$, choose one $\Gamma$-geodesic $P_{xy}$ from $x$ to $y$. 
We first enlarge $H$ by adding all these geodesics. 
Namely, let $Y$ be the subgraph of $\Gamma$ obtained from $H$ by adding all paths $P_{xy}$.
Since $H$ is connected, for every $x,y\in H$ we have
    \[
    \dist_\Gamma(x,y)\le \dist_{H}(x,y)\le n-1.
    \]
Thus each path $P_{xy}$ has at most $n$ vertices. 
Therefore
    \[
    |V(Y)|
    \le n+\binom n2 n
    \le n^3.
    \]
Moreover, $Y$ is connected, since it contains the connected graph $H$ and
    each added path $P_{xy}$ has both endpoints in $H$.
    We claim the following:
    %\begin{clm}\label{clm:quasi-convex}
    Every $\Gamma$-geodesic between two vertices of $Y$ lies in the
    $8\delta $-neighbourhood of $Y$.
    %\end{clm}
    %\begin{clmproof}
    Let $u,v\in V(Y)$, and let $Q$ be a $\Gamma$-geodesic from $u$ to $v$. 
    By the construction of $Y$, there exist $a,b,c,d\in V(H)$ such that $u\in P_{ab}$, and $v\in P_{cd}$.
    Here, if necessary, we allow the trivial path $P_{aa}$ when $u=a\in H$, and similarly for $v$.
    Consider the geodesic quadrilateral with sides
    $[u,a]\subseteq P_{ab}$, $P_{ad}$, $[d,v]\subseteq P_{cd}$, and $Q$, see Figure~\ref{fig:F}.
    By Lemma~\ref{lem:thin}, geodesic quadrilaterals in $\Gamma$ are
    $8\delta$-slim. Hence every vertex of $Q$ is within distance at most $8\delta $ from
    \[
    P_{ab}\cup P_{ad}\cup P_{cd}\subseteq Y.
    \]
    Therefore the claim is proved.
    %\end{clmproof}

    Now we perform a second enlargement. 
    Let $F$ be the subgraph of $\Gamma$ obtained from $Y$ as follows: for every unordered pair of distinct vertices
    $y,z\in V(Y)$ with $\dist_\Gamma(y,z)\le 16\delta +1$,
    add one chosen $\Gamma$-geodesic from $y$ to $z$. 
    Then $F$ is finite and connected, and $H\subseteq Y\subseteq F$.
    Indeed, $Y$ is connected and $F$ is obtained from $Y$ by adding paths.
    
    We next bound the number of vertices of $F$. There are at most $|V(Y)|^2$ pairs of vertices of $Y$, and each geodesic added in the second step has at most $16\delta +2$
    vertices. Hence
    \[
    \begin{aligned}
    |V(F)|
    &\le |V(Y)|
       +\bigl(16\delta +2\bigr)|V(Y)|^2 \\
    &\le \bigl(16\delta +3\bigr)|V(Y)|^2 \\
    &\le \bigl(16\delta +3\bigr)n^6.
    \end{aligned}
    \]

    It remains to prove that $F$ is uniformly hyperbolic. We show that the inclusion $\iota\colon F\hookrightarrow \Gamma$
    is a quasi-isometric embedding with constants depending only on $\delta$.
    First let $y,z\in V(Y)$. 
    Let $y=p_0,p_1,\ldots,p_m=z$ be a $\Gamma$-geodesic from $y$ to $z$, so that $m=\dist_\Gamma(y,z)$. 
    By the above claim, for each $i$ we can choose
    $y_i\in V(Y)$ such that $\dist_\Gamma(p_i,y_i)\le 8\delta$,
    with $y_0=y$ and $y_m=z$. 
    Then
    \[
    \dist_\Gamma(y_i,y_{i+1})
    \le
    8\delta +1+8\delta 
    =
    16\delta +1.
    \]
    Therefore, by construction, the graph $F$ contains a path from $y_i$ to $y_{i+1}$ of length at most $16\delta +1$. 
    Concatenating these paths gives
    \[
    \dist_F(y,z)
    \le
    \bigl(16\delta +1\bigr)\dist_\Gamma(y,z)
    \qquad
    \text{for all } y,z\in V(Y).
    \]
    
    Now let $p,r\in V(F)$ be arbitrary. Since every vertex of $F$ either lies in
    $Y$ or lies on one of the geodesics added in the second step, there exist
    vertices $y,y'\in V(Y)$ such that
    \[
    \dist_F(p,y)\le 16\delta +1,
    \qquad
    \dist_F(r,y')\le 16\delta +1.
    \]
    Using the estimate above for vertices of $Y$, we get
    \[
    \begin{aligned}
    \dist_F(p,r)
    &\le \dist_F(p,y)+\dist_F(y,y')+\dist_F(y',r) \\
    &\le 2\bigl(16\delta +1\bigr)
       +\bigl(16\delta +1\bigr)\dist_\Gamma(y,y') \\
    &\le 2\bigl(16\delta +1\bigr)
       +\bigl(16\delta +1\bigr)
         \bigl(\dist_\Gamma(p,r)+2(16\delta +1)\bigr) \\
    &=
    \bigl(16\delta +1\bigr)\dist_\Gamma(p,r)
    +2\bigl(16\delta +1\bigr)^2
    +2\bigl(16\delta +1\bigr).
    \end{aligned}
    \]
    The reverse inequality $\dist_\Gamma(p,r)\le \dist_F(p,r)$ holds because $F$ is a subgraph of $\Gamma$. Hence the inclusion
    $\iota\colon F\hookrightarrow \Gamma$ is a quasi-isometric embedding with constants
    depending only on $\delta$.
    By Lemma~\ref{lem:qi-hyp}, since $\Gamma$ is $\delta$-hyperbolic, the graph $F$
    is $\Delta_\delta$-hyperbolic for some constant $\Delta_\delta$ depending
    only on $\delta$.
    % 
    % Finally, if $G$ is $H$-minor-free, then every subgraph of $G$ is
    % $H$-minor-free. Since $F\subseteq G$, it follows that $F$ is
    % $H$-minor-free.
\end{proof}

\begin{figure}
    \centering
    \includegraphics[width=0.5\linewidth]{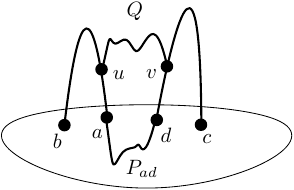}
    \caption{Geodesic quadrilateral}
    \label{fig:F}
\end{figure}

\subsection{(Weighted) separators in apex-minor-free graphs}

The following theorem of Coudert, Ducoffe, and Nisse is the structural input
from the theory of apex-minor-free graphs.

\begin{theorem}[{\cite[Thm.~4.9]{CoudertDucoffeNisse2016}}]
\label{lem:tw-tl-apex}
Let $A$ be a fixed apex graph. Then there exists a constant $c_A>0$,
depending only on $A$, such that for every connected $A$-minor-free graph $\Gamma$, we have $\tw(\Gamma)\le 3c_A\bigl(\tl(\Gamma)+1\bigr)$.
\end{theorem}

The next two lemmas explain how to turn bounded treewidth into balanced separators.  
We first use the elementary centroid property of weighted
trees, and then apply it to the decomposition tree of a graph.

\begin{lemma}{\rm \cite[\S3, Lem.~ 1]{BielakPanczyk2012}}\label{lem:centroid}
Let $T$ be a finite tree and let $\mu\colon V(T)\to \mathbb R^{\ge 0}$ be a vertex weight function. Then there exists a vertex $x\in V(T)$ such that for every connected component $C$ of $T-x$, we have $\mu(V(C))\le \frac12 \mu(V(T))$.
\end{lemma}

\begin{lemma}
\label{lem:weighted-tw-sep}
Let $H$ be a finite graph, let $(T,\beta)$ be a tree-decomposition of $H$, and let $\mu\colon V(H)\to \mathbb R^{\ge 0}$ be a weight
function. Then there exists a node $x\in V(T)$ such that, 
every connected component $K$ of $H-\beta(x)$ satisfies $\mu(K)\le \frac12\mu(V(H))$.
\end{lemma}

\begin{proof}
For each $v\in V(H)$, choose a node $\alpha(v)\in V(T)$ such that
$v\in\beta(\alpha(v))$. Define a weight function $\mu_T$ on $V(T)$ by
\[
    \mu_T(x)\coloneqq \sum_{\alpha(v)=x}\mu(v).
\]
By \Cref{lem:centroid}, there exists $x\in V(T)$ such
that every connected component $C$ of $T-x$ satisfies
\[
    \mu_T(V(C))\le \frac12\mu_T(V(T))
    =
    \frac12\mu(V(H)).
\]
Then we set  $Z\coloneqq \beta(x)$.
Let $K$ be a connected component of $H-Z$. For every $v\in V(K)$, the subtree $T_v\coloneqq \beta^{-1}(v)$  does not contain $x$, since $v\notin Z=\beta(x)$. Hence $T_v$ is contained
in a component of $T-x$.
If $u,v\in V(K)$ are adjacent in $H$, then some bag contains both $u$ and
$v$, so $T_u\cap T_v\neq\emptyset$. Therefore $T_u$ and $T_v$ lie in the
same component of $T-x$. Since $K$ is connected, all subtrees $T_v$ with
$v\in V(K)$ lie in one component $C$ of $T-x$. Thus
\[
    \mu(K)\le \mu_T(V(C))\le \frac12\mu(V(H)).\qedhere
\]
\end{proof}

Combining the preceding ingredients gives the finite weighted separator statement that will be used in the proof of the main theorem.  
Hyperbolicity gives logarithmic tree-length, the apex-minor-free assumption converts this to logarithmic treewidth, and the weighted separator lemma then produces a balanced separator of logarithmic size.

\begin{corollary}
\label{lem:finite-weighted}
    Let $A$ be a fixed apex graph. There exists a constant $b_A$ such that every finite connected $A$-minor-free $\delta$-hyperbolic graph $F$ on $N$ vertices
    has the following property: for every vertex weight function
    $\mu\colon V(F)\to\mathbb R^{\ge 0}$, there is $Z\subseteq V(F)$ such that $|Z|\le b_A(1+\delta\log_2 N)$, and every connected component $K$ of $F-Z$ satisfies 
    $$\mu(K)\le \frac{\mu(V(F))}{2}.$$
\end{corollary}

\begin{proof}
    It follows from~\Cref{lem:tl-from-gromov}, that  $\tl(F)\le  1+2\delta\log_2 N$.
    Next we invoke \Cref{lem:tw-tl-apex}, and so we know that there exists a constant $c_A>0$,
    depending only on $A$, such that  $\tw(F)\le 3c_A\bigl(\tl(F)+1\bigr)$ which implies that
    $\tw(F)\le 3c_A (2+2\delta\log_2 N)$.
    Now~\Cref{lem:weighted-tw-sep} finishes the proof.
\end{proof}

\subsection{Deducing Theorem~\ref{thm:main-apex}}

We almost have all we need to deduce Theorem~\ref{thm:main-apex} of the introduction. 
The final hurdle is that our hull construction relies heavily on the fact that the subgraphs of our ambient graph we are considering are connected. 
The following elementary reduction shows that this is enough: if every connected finite set has a logarithmic separator, then the same is true for arbitrary finite sets, after changing the constant.
    
\begin{lemma}
\label{lem:connected-reduction}
Let $\Gamma$ be a graph.
Suppose there is a constant $C$ such that every finite connected subgraph $H\subseteq \Gamma$ satisfies $\cut(H)\le C\log_2(|V(H)|+1)$.
Then $\sep_\Gamma(n)\le C\log_2(n+1)$ for all $n\ge 1$.
\end{lemma}
    
\begin{proof}
Fix $n \geq 1$ and let $H\subseteq \Gamma$ be a finite subgraph on  $N :=|V(H)|\le n$ vertices.
If every connected component of $H$ has at most $N/2$ vertices, then $\cut(H)=0$.

Otherwise, there is a unique connected component $H_0$ of $H$ with
$|V(H_0)|>N/2$. By hypothesis, choose $C_0\subseteq V(H_0)$ with
\[
    |C_0|\le C\log_2(|V(H_0)|+1)\le C\log_2(n+1)
\]
such that every connected component of $H_0-C_0$ has at most
$|V(H_0)|/2\le N/2$ vertices.
All other connected components of $H$ have fewer than $N/2$ vertices.
Therefore every connected component of $H-C_0$ has at most $N/2$ vertices.
Hence we have 
\[
    \cut(H)\le |C_0|\le C\log_2(n+1).
\]
Taking the maximum over all finite subgraphs $H\subseteq \Gamma$ with
$|V(H)|\le n$ gives the result.
\end{proof}

We now proceed to deduce Theorem~\ref{thm:main-apex}.
The proof applies the hull construction to a connected finite subgraph $H$, uses the finite weighted separator corollary on the resulting graph $F$, and then intersects the separator back with $H$.

\mainapex*

\begin{proof}
By~\Cref{lem:connected-reduction}, it is enough to prove the desired logarithmic bound for finite connected subgraphs of $\Gamma$.
Let $H\subseteq \Gamma$ be a finite connected subgraph, and put $n\coloneqq |V(H)|$.
The case $n=1$ is trivial, so assume $n\ge 2$.
Apply~\Cref{lem:hull} to the connected finite subgraph $H\subseteq \Gamma$.
Then there is a finite connected subgraph $F\subseteq \Gamma$ such that
\[
    H\subseteq F,
    \qquad
    |V(F)|\le M_\delta n^6,
\]
where $M_\delta\coloneqq 16\delta+3$, and such that $F$ is
$\Delta_\delta$-hyperbolic for some constant $\Delta_\delta$ depending only on
$\delta$. Since $F\subseteq \Gamma$ and $\Gamma$ is $A$-minor-free, the graph
$F$ is also $A$-minor-free.
Define a weight function $\mu\colon V(F)\to\{0,1\}$ by
\[
\mu(v)=
\begin{cases}
1, & v\in V(H),\\
0, & v\notin V(H).
\end{cases}
\]
Then we have $\mu(V(F))=|V(H)|=n$.
By~\Cref{lem:finite-weighted}, applied to the finite connected $A$-minor-free $\Delta_\delta$-hyperbolic graph $F$, there exists a set $Z\subseteq V(F)$ such that every connected component $K$ of $F-Z$ satisfies
\[
    \mu(K)\le \frac n2,
\]
and we have
\[
    |Z|
    \le
    b_A\bigl(1+\Delta_\delta\log_2 |V(F)|\bigr).
\]
Using $|V(F)|\le M_\delta n^6$, we obtain
\[
\begin{aligned}
    |Z|
    &\le
    b_A\bigl(1+\Delta_\delta\log_2(M_\delta n^6)\bigr)  \\
    &=
    b_A\bigl(1+\Delta_\delta\log_2 M_\delta
        +6\Delta_\delta\log_2 n\bigr)                  \\
    &\le
    C_1(A,\delta)\log_2(n+1),
\end{aligned}
\]
for some constant $C_1(A,\delta)$ depending only on $A$ and $\delta$.

Now we set $C_H\coloneqq Z\cap V(H)$.
Since $H\subseteq F$, the graph $H-C_H$ is a subgraph of $F-Z$. 
Hence every connected component of $H-C_H$ is contained in a connected component of $F-Z$. 
If $L$ is a connected component of $H-C_H$, and $K$ is the connected component of $F-Z$ containing $L$, then
\[
    |V(L)|=\mu(L)\le \mu(K)\le \frac n2.
\]
Therefore every connected component of $H-C_H$ has at most $n/2$ vertices.
Thus $C_H$ is a balanced separator of $H$, and so
\[
    \cut(H)\le |C_H|\le |Z|
    \le C_1(A,\delta)\log_2(n+1).
\]

Since this holds for every finite connected subgraph $H\subseteq \Gamma$,
\Cref{lem:connected-reduction} gives the following:
\[
    \sep_\Gamma(n)\le C(A,\delta)\log_2(n+1)
\]
for all $n\ge 1$, after increasing the constant if necessary.
\end{proof}

%%%%%%%%%%%%%%%%%%%%%%%%%%%%%%%%%%%%%%%%%%%%%%%%%%%%%%%%%%%%%%%%%%%%%%%%%%%%%%%%%%%%%%%%%%%%
%%%%%%%%%%%%%%%%%%%%%%%%%%%%%%%%%%%%%%%%%%%%%%%%%%%%%%%%%%%%%%%%%%%%%%%%%%%%%%%%%%%%%%%%%%%%
%%%%%%%%%%%%%%%%%%%%%%%%%%%%%%%%%%%%%%%%%%%%%%%%%%%%%%%%%%%%%%%%%%%%%%%%%%%%%%%%%%%%%%%%%%%%

\section{Separation profiles via grids}

In this section we prove Theorem~\ref{thm:main-planar} of the introduction. The argument involves studying the presence of grid minors in our subgraphs, and relating the observations made back to the tree-decompositions witnessed in the previous section.

\subsection{The grid profile}\label{sec:gridprofile}

To calculate the separation profile of our hyperbolic planar graphs, we will make use of an alternative characterisation in terms of grids. Let $Q_n$ denote the $n \times n$ square grid; that is, the grid containing $n^2$ vertices. We have the following definition. 

\begin{definition}[Grid profile]
    Let $\Gamma$ be a graph. We write
     $$
     \maxgrid(\Gamma) := \sup\{n \in \N: Q_n \prec \Gamma\}. 
     $$
    Then, define the \emph{grid profile} of $\Gamma$, denoted $\grid_\Gamma : \N \to \N$, as
    $$
    \grid_\Gamma(n) := \max\Big\{ \maxgrid(H) : \text{$H \subset \Gamma$, $|V(H)| \leq n$}\Big\}.
    $$
\end{definition}

For planar graphs, the grid profile is asymptotically equivalent to the separation profile, via the following result. 

\begin{proposition}\label{prop:grid-sep-agree}
    Let $\Gamma$ be a planar graph. Then 
    $$
    \frac 1{15} \grid_\Gamma(n) \leq \sep_\Gamma(n)  \leq 5\grid_\Gamma(n)
    $$
    for all $n \geq 1$.
\end{proposition}

\begin{proof}
    Denote by $\tw_\Gamma : \N \to \N$ the \emph{tree-width profile} of $\Gamma$. That is, the function
    $$
    \tw_\Gamma(n) := \max\{\tw(H) : H \subset \Gamma, |V(H)| \leq n\}.
    $$
    It was recently highlighted by Hume \cite{hume2026separation} that $\tw_\Gamma \simeq \sep_\Gamma$, a fact which follows from a theorem of Dvo{\v{r}}{\'a}k--Norin \cites{dvovrak2019treewidth,houdrouge2025separation}. More precisely, we have that 
    $$
    \sep_\Gamma(n)-1 \leq \tw_\Gamma(n) \leq 15 \sep_\Gamma(n). 
    $$
    Since $\tw(Q_n) = n$, it is easy to see that $\tw(H) \geq \maxgrid(H)$ for all graphs $H$. Conversely, applying a bound of Grigoriev \cite{grigoriev2011tree} (see also an earlier linear bound of Seymour--Thomas \cite{seymour1994call}), if $H$ is planar and $\maxgrid(H) < m$, then $\tw(H) \leq 5m-6$. In particular, we deduce that
    $$
    \grid_\Gamma(n) \leq \tw_\Gamma(n) \leq 5\grid_\Gamma(n)-1
    $$
    for all $n \geq 1$. 
    Rearranging, we conclude that
    $$
    \frac 1{15} \grid_\Gamma(n) \leq \sep_\Gamma(n)  \leq 5\grid_\Gamma(n).
    $$
    The proposition follows. 
\end{proof}

\begin{remark}
    Note that the lower bound in Proposition~\ref{prop:grid-sep-agree} holds without the planarity hypothesis, which was only used to deduce the upper bound. 
    For more general graphs, a non-linear upper bound can still be established. A theorem of Chuzhoy and Tan \cite{chuzhoy2021towards} asserts that if a graph $H$ excludes an $m \times m$ grid as a minor, then
    $$
    \tw(H) = \cO (m^9 \operatorname{polylog}(m)).
    $$
    In particular, one deduces that for an arbitrary graph we have the asymptotic bound
    $$
    \sep_\Gamma(n) = \cO (\grid_\Gamma(n)^9 \operatorname{polylog}(\grid_\Gamma(n))).
    $$
\end{remark}

\subsection{Grids and geodesic loaded cycles}

In this section we work with \emph{geodesic loaded cycles}, a tool which was introduced by Berger and Seymour in \cite{berger2024bounded} to study the diameters of bags in tree-decompositions. We will use the presence of grid minors in a planar graph to explicitly construct examples of geodesic loaded cycles.

\begin{definition}[Geodesic loaded cycle]
    Let $\Gamma$ be a graph. A \emph{loaded cycle} in $\Gamma$ is an ordered pair $(C, F)$ where:
    \begin{enumerate}
        \item $C \subset \Gamma$ is a simple cycle, and 

        \item $F \subset E(C)$ is a collection of edges of $C$, called \emph{marked edges}.  
    \end{enumerate}
    Given $u, v \in V(C)$, we define their \emph{loaded distance} in $(C,F)$ via
    $$
    \dist_{C,F}(u,v) := \min\{|E(P) \cap F| : \text{$P \subset C$  is one of the two $u$-$v$ arcs in $C$}\}.
    $$
    We say that $(C,F)$ is a \emph{geodesic loaded cycle} (\emph{GLC} for short), if 
    $$
    \dist_{C,F}(u,v) \leq \dist_\Gamma(u,v)
    $$
    for all $u,v \in V(C)$. The \emph{load} of $(C,F)$ is defined to be the quantity $|F|$. 
\end{definition}

Berger and Seymour proved the following key lemma. 

\begin{lemma}[{\cite[Lem.~2.3]{berger2024bounded}}]\label{lem:berger-seymour}
    Let $\Gamma$ be a graph and $(C,F)$ a loaded cycle, with $|F| \geq 2$. Let $(T,\beta)$ be a tree-decomposition of $\Gamma$. Then there exists $v \in V(T)$ and $u,w \in V(C) \cap \beta(v)$ such that 
    $$
    \dist_{C,F}(u,w) \geq |F|/3.
    $$
\end{lemma}

Applying Lemma~\ref{lem:berger-seymour} to geodesic loaded cycles provides an easy route to proving lower bounds for the diameters of bags in tree-decompositions.

We will now prove the following result which is the technical crux of this proof. This lemma will enable us to use grids in planar graphs to construct large GLCs.

\begin{lemma}\label{lem:grid-GLC-lemma}
    Let $n \geq 14$. Let $\Gamma$ be a planar graph and $H \subset \Gamma$ a subgraph such that
    $$
    \maxgrid(H) \geq n.
    $$
    Then $\Gamma$ contains a GLC $(C,F)$ satisfying $C \subset H$ and $|F| \geq n/4$.
\end{lemma}

\begin{proof}
    Fix $n \geq 14$. 
    Consider an $n \times n$ grid minor in $H$. We will realise this minor using a \emph{model}. That is, we have the following data:
    \begin{enumerate}
        \item For every $v \in V(Q_n)$, a subtree $B_v \subset H$, called a \emph{branch set}, such that all branch sets are pairwise disjoint.

        \item If $u,v \in V(Q_n)$ are adjacent, then there exists an edge $e \in E(H)$ with one endpoint in $B_v$ and the other in $B_u$. 
    \end{enumerate}
    Without loss of generality, let us assume that $V(H) = \bigcup_v B_v$, simply throwing away the pieces of $H$ which do not contribute to the grid minor.
    This yields a natural map $ \iota : V(H) \to V(Q_n)$. We may identify 
    $$
    V(Q_n) \equiv \{1, \ldots, n\} \times \{1, \ldots, n\}.
    $$
    It will be convenient to work with the square grid with the square diagonals included, i.e. the `King's graph'. We denote this graph by $\widehat Q_n$, identifying $V(Q_n)$ and $V(\widehat Q_n)$ in the obvious way. 
    Note that the path metric in $\widehat Q_n$ is precisely the $\ell_\infty$ metric $\rho$, given by
    $$
    \rho((a,b),(c,d)) = \max\{|a-c|,|b-d|\}.
    $$
    Slightly abusing notation, we pull this back to a pseudo-metric on $V(H)$, writing $\rho(x,y) := \rho(\iota(x),\iota(y))$.

    Note that the drawing of $H$ will correspond precisely to the standard drawing of the grid, in terms of the way cycles nest. Indeed, since the grid is a 3-connected planar graph (up to collapsing the degree-2 vertices at the corners), by Whitney's planar embedding theorem the drawing of $Q_n$ is unique. Contracting the edges of the branch sets recovers a drawing of the grid, and so our drawing of $H$ cannot be anything `strange'. 
    
    We now choose our cycle $C \subset H$. This will be some choice of cycle deep in the `interior' of the grid. 
    More precisely, consider the sub-grid of $Q_n$ spanned by the vertices in
    $$
    \{\lfloor n/3 \rfloor+1, \ldots, \lfloor2n/3 \rfloor\} \times \{\lfloor n/3 \rfloor+1, \ldots, \lfloor2n/3 \rfloor\}.
    $$
    Let $C'$ denote the boundary cycle of this sub-grid. 
    Note that $C'$ contains exactly $4\lfloor\tfrac{n-2}{3} \rfloor$ edges, and in particular is non-trivial whenever $n \geq 5$. 
    Pull $C'$ back to some fixed choice of simple cycle $C$ in $H$ which passes through the branch sets in the same cyclic order with which $C'$ passes through its vertices. Since the branch sets are chosen to be trees, this cycle $C$ is uniquely determined. 
    We now have the following claim, which motivates why we chose our cycle to lie `deep inside' the grid.

    \begin{claim}\label{claim:ambient-distance}
        For all $u, v \in V(C)$, we have $\dist_\Gamma(u,v) \geq \rho(u,v)$.
    \end{claim}

    \begin{proof}
        Consider a geodesic $\gamma$ in $\Gamma$ connecting $u$ and $v$. If $\gamma$ passes through the outer-most boundary layer of the grid, we automatically have that 
        $$
        \length(\gamma) > \diam_\rho(V(C)),
        $$ 
        by the Jordan curve theorem, since $\gamma$ will have to `cross' at least this many disjoint cycles (twice each) to reach the outside and then return to $C$. Thus, we may assume without loss of generality that $\gamma$ stays within the region bounded by the boundary of the grid. Under this assumption, note that if $\gamma$ leaves a branch set $B_z$, then the next branch it enters must necessarily be of the form $B_x$ where $x$ is adjacent to $z$ in $\widehat Q_n$. This also follows immediately from the Jordan curve theorem. From this, the claim follows, as $\gamma$ projects to a path in $\widehat Q_n$ between $\iota(u)$ and $\iota(v)$ of length at most $\length(\gamma)$. 
    \end{proof}

    We now choose the set of marked edges $F \subset E(C)$. 
    Consider the `top side' of square cycle $C'$ in $Q_n$ which we used earlier to define $C$. Note that every edge in $C'$ corresponds to a unique edge in $C$. Let $F'$ denote the edges on the `top side' of $C'$, and let $F$ denote precisely the edges of our model of the grid which correspond to the edges in $F'$. Note that
    $$
    |F| = \left\lfloor\frac{n-2}{3} \right\rfloor \geq \frac n 4,
    $$
    for all $n \geq 14$. 
    We now observe that the following bound holds.
    
    \begin{claim}\label{claim:loaded-distance}
        For all $u, v \in V(C)$, we have $\dist_{C,F}(u,v) \leq \rho(u,v)$.
    \end{claim}

    \begin{proof}
        It is immediate that 
        $
        \dist_{C,F}(u,v) = \dist_{C',F'}(\iota(u), \iota(v))
        $. 
        Indeed, the number of marked edges on the two $u$-$v$ paths in $C$ is equal to the number of marked edges on the two $\iota(u)$-$\iota(v)$ paths in $C'$. 
        Thus, it is sufficient to work entirely within the grid itself. For simplicity, we will abuse notation and write $\iota(u) \equiv u$ and $\iota(v) \equiv v$. 
        Note also that 
        $$
        \dist_{C',F'}(u, v) \leq \left\lfloor\frac{|F|}2\right\rfloor.
        $$
        If neither $u$ nor $v$ lie on the top side, then $\dist_{C',F'}(u,v) = 0$ and there is nothing to prove. Suppose without loss of generality that $u$ lies on the top side of $C'$.
        We have a few cases to consider. 
        Suppose that $v$ lies on the bottom side of $C'$. Then clearly $\rho(u,v) \geq |F|$, and we are done. Suppose that $v$ also lies on the top side. Then there is certainly a path in $C'$ from $v$ to $u$ which passes through $\rho(u,v)$ marked edges, and we are also done. Finally, suppose that $v$ lies on the left or right side of $C'$. Then, we again have that there is certainly a path from $u$ to $v$ which passes through at most $\rho(u,v)$ marked edges. The claim follows. 
    \end{proof}

    Combining Claims~\ref{claim:ambient-distance} and \ref{claim:loaded-distance}, it is immediate that $(C,F)$ is a GLC. The lemma follows. See Figure~\ref{fig:grid-GLC} for a cartoon of the construction.
\end{proof}

\begin{figure}
    \centering
    \input{grid-lemma}
    \caption{The GLC constructed within a planar grid minor in Lemma~\ref{lem:grid-GLC-lemma}. The vertices of this grid drawing denote arbitrary branch sets.}
    \label{fig:grid-GLC}
\end{figure}
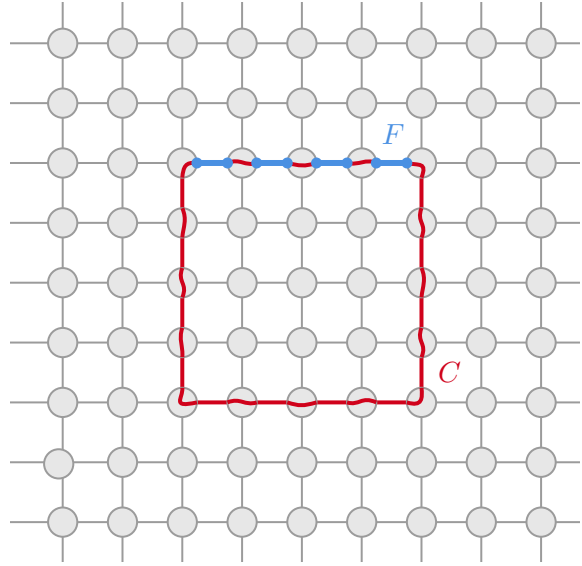

Combining the above with Lemma~\ref{lem:berger-seymour}, we can deduce the following.

\begin{lemma}\label{lem:tree-decompositions-and-grids}
    Let $n \geq 14$. Let $\Gamma$ be a planar graph and $H \subset \Gamma$ a subgraph such that
    $
    \maxgrid(H) \geq n
    $. 
    Let $(T, \beta)$ be a tree-decomposition of $H$. Then there exists $v \in V(T)$ such that
    $$
    \diam_\Gamma(\beta(v)) \geq n/12. 
    $$
\end{lemma}

\begin{proof}
    By Lemma~\ref{lem:grid-GLC-lemma}, there exists a GLC $(C,F)$ such that $C \subset H$ and $|F| \geq n/4$. By Lemma~\ref{lem:berger-seymour}, there exists $v \in V(T)$ and $u,w \in V(C) \cap \beta(v)$ such that 
    $$
    n/12\leq |F|/3 \leq \dist_{C,F}(u,w) \leq \dist_\Gamma(u,w).
    $$
    Thus, we are done.
\end{proof}

\subsection{Deducing Theorem~\ref{thm:main-planar}}

We now have all we need to conclude the proof of Theorem~\ref{thm:main-planar}.

\mainplanar*

\begin{proof}
    Let $\Gamma$ be a $\delta$-hyperbolic planar graph. Let us assume without loss of generality that $\Gamma$ is connected. Fix $n \in \N$, and let $H \subset \Gamma$ be some subgraph satisfying $|V(H)| \leq n$. We will bound $\maxgrid(H)$. Write $m = \maxgrid(H)$. Assume without loss of generality that $m \geq 14$, since the desired statement is trivial if $m < 14$. 
    By Proposition~\ref{prop:tree-decomp-from-tree-approx}, there exists a tree-decomposition $(T,\beta)$ of $H$ such that
    $$
    \diam_\Gamma(\beta(v)) \leq 1 + 2\delta \log_2(|V(H)|),
    $$
    for all $v \in V(T)$.
    However, by Lemma~\ref{lem:tree-decompositions-and-grids}, we have that there exists $v \in V(T)$ such that 
    $
    \diam_\Gamma(\beta(v)) \geq m/12
    $. 
    It follows that 
    $
    m \leq  12 + 24\delta \log_2(|V(H)|)
    $. 
    In particular, we deduce that 
    $$
    \grid_\Gamma(n) \leq  12 + 24\delta \log_2(n).
    $$
    Applying Proposition~\ref{prop:grid-sep-agree}, we deduce that
    $$
    \sep_\Gamma(n) \leq   60 + 120\delta \log_2(n).
    $$
    Thus, we are done.
\end{proof}

\bibliographystyle{abbrv}
\bibliography{references}

\end{document}

%% file: commands.tex
\DeclareMathOperator{\diam}{diam}
\DeclareMathOperator{\length}{length}
\DeclareMathOperator{\cut}{cut}

\DeclareMathOperator{\dist}{d}

\DeclareMathOperator{\width}{width}

\DeclareMathOperator{\sep}{sep}
\DeclareMathOperator{\grid}{grid}
\DeclareMathOperator{\maxgrid}{maxgrid}
\DeclareMathOperator{\tw}{tw}
\DeclareMathOperator{\tl}{tl}

\newcommand{\R}{\mathbb{R}}

\newcommand{\N}{\mathbb{N}}

\newcommand{\cO}{\mathcal{O}}

\makeatletter
\newcommand{\myitem}[1]{%
\item[#1]\protected@edef\@currentlabel{#1}%
}
\makeatother

\newtheorem{theorem}{Theorem}[section]% theorem counter resets every \subsection
\newtheorem{alphtheorem}{Theorem}

\newtheorem{alphcor}[alphtheorem]{Corollary}
\newtheorem*{theorem*}{Theorem}

\newtheorem{proposition}[theorem]{Proposition}
\newtheorem{lemma}[theorem]{Lemma}
\newtheorem{corollary}[theorem]{Corollary}
\newtheorem{claim}[theorem]{Claim}

\theoremstyle{definition}
\newtheorem{definition}[theorem]{Definition}

\newtheorem{remark}[theorem]{Remark}

%% file: grid-lemma.tex
\tikzset{every picture/.style={line width=0.75pt}} %set default line width to 0.75pt        

\begin{tikzpicture}[x=0.75pt,y=0.75pt,yscale=-1,xscale=1]
%uncomment if require: \path (0,353); %set diagram left start at 0, and has height of 353

%Shape: Grid [id:dp8376667742170463] 
\draw  [draw opacity=0] (171.4,30.2) -- (471.4,30.2) -- (471.4,330.2) -- (171.4,330.2) -- cycle ; \draw  [color={rgb, 255:red, 155; green, 155; blue, 155 }  ,draw opacity=1 ] (201.4,30.2) -- (201.4,330.2)(231.4,30.2) -- (231.4,330.2)(261.4,30.2) -- (261.4,330.2)(291.4,30.2) -- (291.4,330.2)(321.4,30.2) -- (321.4,330.2)(351.4,30.2) -- (351.4,330.2)(381.4,30.2) -- (381.4,330.2)(411.4,30.2) -- (411.4,330.2)(441.4,30.2) -- (441.4,330.2) ; \draw  [color={rgb, 255:red, 155; green, 155; blue, 155 }  ,draw opacity=1 ] (171.4,60.2) -- (471.4,60.2)(171.4,90.2) -- (471.4,90.2)(171.4,120.2) -- (471.4,120.2)(171.4,150.2) -- (471.4,150.2)(171.4,180.2) -- (471.4,180.2)(171.4,210.2) -- (471.4,210.2)(171.4,240.2) -- (471.4,240.2)(171.4,270.2) -- (471.4,270.2)(171.4,300.2) -- (471.4,300.2) ; \draw  [color={rgb, 255:red, 155; green, 155; blue, 155 }  ,draw opacity=1 ] (171.4,30.2) -- (471.4,30.2) -- (471.4,330.2) -- (171.4,330.2) -- cycle ;
%Shape: Rectangle [id:dp06911119666036347] 
\draw  [color={rgb, 255:red, 208; green, 2; blue, 27 }  ,draw opacity=1 ][line width=1.5]  (261.4,120.2) -- (381.4,120.2) -- (381.4,240.2) -- (261.4,240.2) -- cycle ;
%Straight Lines [id:da9576956760140773] 
\draw [color={rgb, 255:red, 74; green, 144; blue, 226 }  ,draw opacity=1 ][fill={rgb, 255:red, 231; green, 231; blue, 231 }  ,fill opacity=1 ][line width=2.25]    (261.4,120.2) -- (381.4,120.2) ;
%Shape: Rectangle [id:dp1619187867755335] 
\draw  [draw opacity=0][fill={rgb, 255:red, 255; green, 255; blue, 255 }  ,fill opacity=1 ] (154.28,23.65) -- (175.07,23.65) -- (175.07,330.4) -- (154.28,330.4) -- cycle ;
%Shape: Rectangle [id:dp5051835251687256] 
\draw  [draw opacity=0][fill={rgb, 255:red, 255; green, 255; blue, 255 }  ,fill opacity=1 ] (163.58,19.78) -- (476.27,19.78) -- (476.27,39.6) -- (163.58,39.6) -- cycle ;
%Shape: Circle [id:dp5643127006768552] 
\draw  [color={rgb, 255:red, 155; green, 155; blue, 155 }  ,draw opacity=1 ][fill={rgb, 255:red, 231; green, 231; blue, 231 }  ,fill opacity=1 ] (194.07,60.2) .. controls (194.07,56.15) and (197.35,52.87) .. (201.4,52.87) .. controls (205.45,52.87) and (208.73,56.15) .. (208.73,60.2) .. controls (208.73,64.25) and (205.45,67.53) .. (201.4,67.53) .. controls (197.35,67.53) and (194.07,64.25) .. (194.07,60.2) -- cycle ;
%Shape: Circle [id:dp010185337930883054] 
\draw  [color={rgb, 255:red, 155; green, 155; blue, 155 }  ,draw opacity=1 ][fill={rgb, 255:red, 231; green, 231; blue, 231 }  ,fill opacity=1 ] (224.07,60.2) .. controls (224.07,56.15) and (227.35,52.87) .. (231.4,52.87) .. controls (235.45,52.87) and (238.73,56.15) .. (238.73,60.2) .. controls (238.73,64.25) and (235.45,67.53) .. (231.4,67.53) .. controls (227.35,67.53) and (224.07,64.25) .. (224.07,60.2) -- cycle ;
%Shape: Circle [id:dp972525844869225] 
\draw  [color={rgb, 255:red, 155; green, 155; blue, 155 }  ,draw opacity=1 ][fill={rgb, 255:red, 231; green, 231; blue, 231 }  ,fill opacity=1 ] (254.07,60.2) .. controls (254.07,56.15) and (257.35,52.87) .. (261.4,52.87) .. controls (265.45,52.87) and (268.73,56.15) .. (268.73,60.2) .. controls (268.73,64.25) and (265.45,67.53) .. (261.4,67.53) .. controls (257.35,67.53) and (254.07,64.25) .. (254.07,60.2) -- cycle ;
%Shape: Circle [id:dp07640230537286208] 
\draw  [color={rgb, 255:red, 155; green, 155; blue, 155 }  ,draw opacity=1 ][fill={rgb, 255:red, 231; green, 231; blue, 231 }  ,fill opacity=1 ] (284.07,60.2) .. controls (284.07,56.15) and (287.35,52.87) .. (291.4,52.87) .. controls (295.45,52.87) and (298.73,56.15) .. (298.73,60.2) .. controls (298.73,64.25) and (295.45,67.53) .. (291.4,67.53) .. controls (287.35,67.53) and (284.07,64.25) .. (284.07,60.2) -- cycle ;
%Shape: Circle [id:dp8107436489873223] 
\draw  [color={rgb, 255:red, 155; green, 155; blue, 155 }  ,draw opacity=1 ][fill={rgb, 255:red, 231; green, 231; blue, 231 }  ,fill opacity=1 ] (314.07,60.2) .. controls (314.07,56.15) and (317.35,52.87) .. (321.4,52.87) .. controls (325.45,52.87) and (328.73,56.15) .. (328.73,60.2) .. controls (328.73,64.25) and (325.45,67.53) .. (321.4,67.53) .. controls (317.35,67.53) and (314.07,64.25) .. (314.07,60.2) -- cycle ;
%Shape: Circle [id:dp4173479155386365] 
\draw  [color={rgb, 255:red, 155; green, 155; blue, 155 }  ,draw opacity=1 ][fill={rgb, 255:red, 231; green, 231; blue, 231 }  ,fill opacity=1 ] (344.07,60.2) .. controls (344.07,56.15) and (347.35,52.87) .. (351.4,52.87) .. controls (355.45,52.87) and (358.73,56.15) .. (358.73,60.2) .. controls (358.73,64.25) and (355.45,67.53) .. (351.4,67.53) .. controls (347.35,67.53) and (344.07,64.25) .. (344.07,60.2) -- cycle ;
%Shape: Circle [id:dp27470261337127] 
\draw  [color={rgb, 255:red, 155; green, 155; blue, 155 }  ,draw opacity=1 ][fill={rgb, 255:red, 231; green, 231; blue, 231 }  ,fill opacity=1 ] (374.07,60.2) .. controls (374.07,56.15) and (377.35,52.87) .. (381.4,52.87) .. controls (385.45,52.87) and (388.73,56.15) .. (388.73,60.2) .. controls (388.73,64.25) and (385.45,67.53) .. (381.4,67.53) .. controls (377.35,67.53) and (374.07,64.25) .. (374.07,60.2) -- cycle ;
%Shape: Circle [id:dp2130840824351531] 
\draw  [color={rgb, 255:red, 155; green, 155; blue, 155 }  ,draw opacity=1 ][fill={rgb, 255:red, 231; green, 231; blue, 231 }  ,fill opacity=1 ] (404.07,60.2) .. controls (404.07,56.15) and (407.35,52.87) .. (411.4,52.87) .. controls (415.45,52.87) and (418.73,56.15) .. (418.73,60.2) .. controls (418.73,64.25) and (415.45,67.53) .. (411.4,67.53) .. controls (407.35,67.53) and (404.07,64.25) .. (404.07,60.2) -- cycle ;
%Shape: Circle [id:dp3999609585756765] 
\draw  [color={rgb, 255:red, 155; green, 155; blue, 155 }  ,draw opacity=1 ][fill={rgb, 255:red, 231; green, 231; blue, 231 }  ,fill opacity=1 ] (434.07,60.2) .. controls (434.07,56.15) and (437.35,52.87) .. (441.4,52.87) .. controls (445.45,52.87) and (448.73,56.15) .. (448.73,60.2) .. controls (448.73,64.25) and (445.45,67.53) .. (441.4,67.53) .. controls (437.35,67.53) and (434.07,64.25) .. (434.07,60.2) -- cycle ;
%Shape: Circle [id:dp8192037853621529] 
\draw  [color={rgb, 255:red, 155; green, 155; blue, 155 }  ,draw opacity=1 ][fill={rgb, 255:red, 231; green, 231; blue, 231 }  ,fill opacity=1 ] (194.07,90.2) .. controls (194.07,86.15) and (197.35,82.87) .. (201.4,82.87) .. controls (205.45,82.87) and (208.73,86.15) .. (208.73,90.2) .. controls (208.73,94.25) and (205.45,97.53) .. (201.4,97.53) .. controls (197.35,97.53) and (194.07,94.25) .. (194.07,90.2) -- cycle ;
%Shape: Circle [id:dp7654874546769155] 
\draw  [color={rgb, 255:red, 155; green, 155; blue, 155 }  ,draw opacity=1 ][fill={rgb, 255:red, 231; green, 231; blue, 231 }  ,fill opacity=1 ] (224.07,90.2) .. controls (224.07,86.15) and (227.35,82.87) .. (231.4,82.87) .. controls (235.45,82.87) and (238.73,86.15) .. (238.73,90.2) .. controls (238.73,94.25) and (235.45,97.53) .. (231.4,97.53) .. controls (227.35,97.53) and (224.07,94.25) .. (224.07,90.2) -- cycle ;
%Shape: Circle [id:dp6250768649229528] 
\draw  [color={rgb, 255:red, 155; green, 155; blue, 155 }  ,draw opacity=1 ][fill={rgb, 255:red, 231; green, 231; blue, 231 }  ,fill opacity=1 ] (254.07,90.2) .. controls (254.07,86.15) and (257.35,82.87) .. (261.4,82.87) .. controls (265.45,82.87) and (268.73,86.15) .. (268.73,90.2) .. controls (268.73,94.25) and (265.45,97.53) .. (261.4,97.53) .. controls (257.35,97.53) and (254.07,94.25) .. (254.07,90.2) -- cycle ;
%Shape: Circle [id:dp666951768654127] 
\draw  [color={rgb, 255:red, 155; green, 155; blue, 155 }  ,draw opacity=1 ][fill={rgb, 255:red, 231; green, 231; blue, 231 }  ,fill opacity=1 ] (284.07,90.2) .. controls (284.07,86.15) and (287.35,82.87) .. (291.4,82.87) .. controls (295.45,82.87) and (298.73,86.15) .. (298.73,90.2) .. controls (298.73,94.25) and (295.45,97.53) .. (291.4,97.53) .. controls (287.35,97.53) and (284.07,94.25) .. (284.07,90.2) -- cycle ;
%Shape: Circle [id:dp11171985701996257] 
\draw  [color={rgb, 255:red, 155; green, 155; blue, 155 }  ,draw opacity=1 ][fill={rgb, 255:red, 231; green, 231; blue, 231 }  ,fill opacity=1 ] (314.07,90.2) .. controls (314.07,86.15) and (317.35,82.87) .. (321.4,82.87) .. controls (325.45,82.87) and (328.73,86.15) .. (328.73,90.2) .. controls (328.73,94.25) and (325.45,97.53) .. (321.4,97.53) .. controls (317.35,97.53) and (314.07,94.25) .. (314.07,90.2) -- cycle ;
%Shape: Circle [id:dp2568682357205009] 
\draw  [color={rgb, 255:red, 155; green, 155; blue, 155 }  ,draw opacity=1 ][fill={rgb, 255:red, 231; green, 231; blue, 231 }  ,fill opacity=1 ] (344.07,90.2) .. controls (344.07,86.15) and (347.35,82.87) .. (351.4,82.87) .. controls (355.45,82.87) and (358.73,86.15) .. (358.73,90.2) .. controls (358.73,94.25) and (355.45,97.53) .. (351.4,97.53) .. controls (347.35,97.53) and (344.07,94.25) .. (344.07,90.2) -- cycle ;
%Shape: Circle [id:dp7224937620248015] 
\draw  [color={rgb, 255:red, 155; green, 155; blue, 155 }  ,draw opacity=1 ][fill={rgb, 255:red, 231; green, 231; blue, 231 }  ,fill opacity=1 ] (374.07,90.2) .. controls (374.07,86.15) and (377.35,82.87) .. (381.4,82.87) .. controls (385.45,82.87) and (388.73,86.15) .. (388.73,90.2) .. controls (388.73,94.25) and (385.45,97.53) .. (381.4,97.53) .. controls (377.35,97.53) and (374.07,94.25) .. (374.07,90.2) -- cycle ;
%Shape: Circle [id:dp9232631596240866] 
\draw  [color={rgb, 255:red, 155; green, 155; blue, 155 }  ,draw opacity=1 ][fill={rgb, 255:red, 231; green, 231; blue, 231 }  ,fill opacity=1 ] (404.07,90.2) .. controls (404.07,86.15) and (407.35,82.87) .. (411.4,82.87) .. controls (415.45,82.87) and (418.73,86.15) .. (418.73,90.2) .. controls (418.73,94.25) and (415.45,97.53) .. (411.4,97.53) .. controls (407.35,97.53) and (404.07,94.25) .. (404.07,90.2) -- cycle ;
%Shape: Circle [id:dp7041665914575933] 
\draw  [color={rgb, 255:red, 155; green, 155; blue, 155 }  ,draw opacity=1 ][fill={rgb, 255:red, 231; green, 231; blue, 231 }  ,fill opacity=1 ] (434.07,90.2) .. controls (434.07,86.15) and (437.35,82.87) .. (441.4,82.87) .. controls (445.45,82.87) and (448.73,86.15) .. (448.73,90.2) .. controls (448.73,94.25) and (445.45,97.53) .. (441.4,97.53) .. controls (437.35,97.53) and (434.07,94.25) .. (434.07,90.2) -- cycle ;
%Shape: Circle [id:dp04219930729794352] 
\draw  [color={rgb, 255:red, 155; green, 155; blue, 155 }  ,draw opacity=1 ][fill={rgb, 255:red, 231; green, 231; blue, 231 }  ,fill opacity=1 ] (194.07,120.2) .. controls (194.07,116.15) and (197.35,112.87) .. (201.4,112.87) .. controls (205.45,112.87) and (208.73,116.15) .. (208.73,120.2) .. controls (208.73,124.25) and (205.45,127.53) .. (201.4,127.53) .. controls (197.35,127.53) and (194.07,124.25) .. (194.07,120.2) -- cycle ;
%Shape: Circle [id:dp3592302740155783] 
\draw  [color={rgb, 255:red, 155; green, 155; blue, 155 }  ,draw opacity=1 ][fill={rgb, 255:red, 231; green, 231; blue, 231 }  ,fill opacity=1 ] (224.07,120.2) .. controls (224.07,116.15) and (227.35,112.87) .. (231.4,112.87) .. controls (235.45,112.87) and (238.73,116.15) .. (238.73,120.2) .. controls (238.73,124.25) and (235.45,127.53) .. (231.4,127.53) .. controls (227.35,127.53) and (224.07,124.25) .. (224.07,120.2) -- cycle ;
%Shape: Circle [id:dp25825361117974077] 
\draw  [color={rgb, 255:red, 155; green, 155; blue, 155 }  ,draw opacity=1 ][fill={rgb, 255:red, 231; green, 231; blue, 231 }  ,fill opacity=1 ] (254.07,120.2) .. controls (254.07,116.15) and (257.35,112.87) .. (261.4,112.87) .. controls (265.45,112.87) and (268.73,116.15) .. (268.73,120.2) .. controls (268.73,124.25) and (265.45,127.53) .. (261.4,127.53) .. controls (257.35,127.53) and (254.07,124.25) .. (254.07,120.2) -- cycle ;
%Shape: Circle [id:dp13334635085919444] 
\draw  [color={rgb, 255:red, 155; green, 155; blue, 155 }  ,draw opacity=1 ][fill={rgb, 255:red, 231; green, 231; blue, 231 }  ,fill opacity=1 ] (284.07,120.2) .. controls (284.07,116.15) and (287.35,112.87) .. (291.4,112.87) .. controls (295.45,112.87) and (298.73,116.15) .. (298.73,120.2) .. controls (298.73,124.25) and (295.45,127.53) .. (291.4,127.53) .. controls (287.35,127.53) and (284.07,124.25) .. (284.07,120.2) -- cycle ;
%Shape: Circle [id:dp537738931664517] 
\draw  [color={rgb, 255:red, 155; green, 155; blue, 155 }  ,draw opacity=1 ][fill={rgb, 255:red, 231; green, 231; blue, 231 }  ,fill opacity=1 ] (314.07,120.2) .. controls (314.07,116.15) and (317.35,112.87) .. (321.4,112.87) .. controls (325.45,112.87) and (328.73,116.15) .. (328.73,120.2) .. controls (328.73,124.25) and (325.45,127.53) .. (321.4,127.53) .. controls (317.35,127.53) and (314.07,124.25) .. (314.07,120.2) -- cycle ;
%Shape: Circle [id:dp5143734011969568] 
\draw  [color={rgb, 255:red, 155; green, 155; blue, 155 }  ,draw opacity=1 ][fill={rgb, 255:red, 231; green, 231; blue, 231 }  ,fill opacity=1 ] (344.07,120.2) .. controls (344.07,116.15) and (347.35,112.87) .. (351.4,112.87) .. controls (355.45,112.87) and (358.73,116.15) .. (358.73,120.2) .. controls (358.73,124.25) and (355.45,127.53) .. (351.4,127.53) .. controls (347.35,127.53) and (344.07,124.25) .. (344.07,120.2) -- cycle ;
%Shape: Circle [id:dp822688717644282] 
\draw  [color={rgb, 255:red, 155; green, 155; blue, 155 }  ,draw opacity=1 ][fill={rgb, 255:red, 231; green, 231; blue, 231 }  ,fill opacity=1 ] (374.07,120.2) .. controls (374.07,116.15) and (377.35,112.87) .. (381.4,112.87) .. controls (385.45,112.87) and (388.73,116.15) .. (388.73,120.2) .. controls (388.73,124.25) and (385.45,127.53) .. (381.4,127.53) .. controls (377.35,127.53) and (374.07,124.25) .. (374.07,120.2) -- cycle ;
%Shape: Circle [id:dp8189418923728563] 
\draw  [color={rgb, 255:red, 155; green, 155; blue, 155 }  ,draw opacity=1 ][fill={rgb, 255:red, 231; green, 231; blue, 231 }  ,fill opacity=1 ] (404.07,120.2) .. controls (404.07,116.15) and (407.35,112.87) .. (411.4,112.87) .. controls (415.45,112.87) and (418.73,116.15) .. (418.73,120.2) .. controls (418.73,124.25) and (415.45,127.53) .. (411.4,127.53) .. controls (407.35,127.53) and (404.07,124.25) .. (404.07,120.2) -- cycle ;
%Shape: Circle [id:dp6267312303115923] 
\draw  [color={rgb, 255:red, 155; green, 155; blue, 155 }  ,draw opacity=1 ][fill={rgb, 255:red, 231; green, 231; blue, 231 }  ,fill opacity=1 ] (434.07,120.2) .. controls (434.07,116.15) and (437.35,112.87) .. (441.4,112.87) .. controls (445.45,112.87) and (448.73,116.15) .. (448.73,120.2) .. controls (448.73,124.25) and (445.45,127.53) .. (441.4,127.53) .. controls (437.35,127.53) and (434.07,124.25) .. (434.07,120.2) -- cycle ;
%Shape: Circle [id:dp6824766844321517] 
\draw  [color={rgb, 255:red, 155; green, 155; blue, 155 }  ,draw opacity=1 ][fill={rgb, 255:red, 231; green, 231; blue, 231 }  ,fill opacity=1 ] (194.07,150.2) .. controls (194.07,146.15) and (197.35,142.87) .. (201.4,142.87) .. controls (205.45,142.87) and (208.73,146.15) .. (208.73,150.2) .. controls (208.73,154.25) and (205.45,157.53) .. (201.4,157.53) .. controls (197.35,157.53) and (194.07,154.25) .. (194.07,150.2) -- cycle ;
%Shape: Circle [id:dp2731028550436686] 
\draw  [color={rgb, 255:red, 155; green, 155; blue, 155 }  ,draw opacity=1 ][fill={rgb, 255:red, 231; green, 231; blue, 231 }  ,fill opacity=1 ] (224.07,150.2) .. controls (224.07,146.15) and (227.35,142.87) .. (231.4,142.87) .. controls (235.45,142.87) and (238.73,146.15) .. (238.73,150.2) .. controls (238.73,154.25) and (235.45,157.53) .. (231.4,157.53) .. controls (227.35,157.53) and (224.07,154.25) .. (224.07,150.2) -- cycle ;
%Shape: Circle [id:dp9606001532114778] 
\draw  [color={rgb, 255:red, 155; green, 155; blue, 155 }  ,draw opacity=1 ][fill={rgb, 255:red, 231; green, 231; blue, 231 }  ,fill opacity=1 ] (254.07,150.2) .. controls (254.07,146.15) and (257.35,142.87) .. (261.4,142.87) .. controls (265.45,142.87) and (268.73,146.15) .. (268.73,150.2) .. controls (268.73,154.25) and (265.45,157.53) .. (261.4,157.53) .. controls (257.35,157.53) and (254.07,154.25) .. (254.07,150.2) -- cycle ;
%Shape: Circle [id:dp8618346225480805] 
\draw  [color={rgb, 255:red, 155; green, 155; blue, 155 }  ,draw opacity=1 ][fill={rgb, 255:red, 231; green, 231; blue, 231 }  ,fill opacity=1 ] (284.07,150.2) .. controls (284.07,146.15) and (287.35,142.87) .. (291.4,142.87) .. controls (295.45,142.87) and (298.73,146.15) .. (298.73,150.2) .. controls (298.73,154.25) and (295.45,157.53) .. (291.4,157.53) .. controls (287.35,157.53) and (284.07,154.25) .. (284.07,150.2) -- cycle ;
%Shape: Circle [id:dp029926673017842687] 
\draw  [color={rgb, 255:red, 155; green, 155; blue, 155 }  ,draw opacity=1 ][fill={rgb, 255:red, 231; green, 231; blue, 231 }  ,fill opacity=1 ] (314.07,150.2) .. controls (314.07,146.15) and (317.35,142.87) .. (321.4,142.87) .. controls (325.45,142.87) and (328.73,146.15) .. (328.73,150.2) .. controls (328.73,154.25) and (325.45,157.53) .. (321.4,157.53) .. controls (317.35,157.53) and (314.07,154.25) .. (314.07,150.2) -- cycle ;
%Shape: Circle [id:dp7386466778393096] 
\draw  [color={rgb, 255:red, 155; green, 155; blue, 155 }  ,draw opacity=1 ][fill={rgb, 255:red, 231; green, 231; blue, 231 }  ,fill opacity=1 ] (344.07,150.2) .. controls (344.07,146.15) and (347.35,142.87) .. (351.4,142.87) .. controls (355.45,142.87) and (358.73,146.15) .. (358.73,150.2) .. controls (358.73,154.25) and (355.45,157.53) .. (351.4,157.53) .. controls (347.35,157.53) and (344.07,154.25) .. (344.07,150.2) -- cycle ;
%Shape: Circle [id:dp5222881810221425] 
\draw  [color={rgb, 255:red, 155; green, 155; blue, 155 }  ,draw opacity=1 ][fill={rgb, 255:red, 231; green, 231; blue, 231 }  ,fill opacity=1 ] (374.07,150.2) .. controls (374.07,146.15) and (377.35,142.87) .. (381.4,142.87) .. controls (385.45,142.87) and (388.73,146.15) .. (388.73,150.2) .. controls (388.73,154.25) and (385.45,157.53) .. (381.4,157.53) .. controls (377.35,157.53) and (374.07,154.25) .. (374.07,150.2) -- cycle ;
%Shape: Circle [id:dp8653979249978141] 
\draw  [color={rgb, 255:red, 155; green, 155; blue, 155 }  ,draw opacity=1 ][fill={rgb, 255:red, 231; green, 231; blue, 231 }  ,fill opacity=1 ] (404.07,150.2) .. controls (404.07,146.15) and (407.35,142.87) .. (411.4,142.87) .. controls (415.45,142.87) and (418.73,146.15) .. (418.73,150.2) .. controls (418.73,154.25) and (415.45,157.53) .. (411.4,157.53) .. controls (407.35,157.53) and (404.07,154.25) .. (404.07,150.2) -- cycle ;
%Shape: Circle [id:dp8134895080777859] 
\draw  [color={rgb, 255:red, 155; green, 155; blue, 155 }  ,draw opacity=1 ][fill={rgb, 255:red, 231; green, 231; blue, 231 }  ,fill opacity=1 ] (434.07,150.2) .. controls (434.07,146.15) and (437.35,142.87) .. (441.4,142.87) .. controls (445.45,142.87) and (448.73,146.15) .. (448.73,150.2) .. controls (448.73,154.25) and (445.45,157.53) .. (441.4,157.53) .. controls (437.35,157.53) and (434.07,154.25) .. (434.07,150.2) -- cycle ;
%Shape: Circle [id:dp9657692097643676] 
\draw  [color={rgb, 255:red, 155; green, 155; blue, 155 }  ,draw opacity=1 ][fill={rgb, 255:red, 231; green, 231; blue, 231 }  ,fill opacity=1 ] (194.07,180.2) .. controls (194.07,176.15) and (197.35,172.87) .. (201.4,172.87) .. controls (205.45,172.87) and (208.73,176.15) .. (208.73,180.2) .. controls (208.73,184.25) and (205.45,187.53) .. (201.4,187.53) .. controls (197.35,187.53) and (194.07,184.25) .. (194.07,180.2) -- cycle ;
%Shape: Circle [id:dp6936613412359617] 
\draw  [color={rgb, 255:red, 155; green, 155; blue, 155 }  ,draw opacity=1 ][fill={rgb, 255:red, 231; green, 231; blue, 231 }  ,fill opacity=1 ] (224.07,180.2) .. controls (224.07,176.15) and (227.35,172.87) .. (231.4,172.87) .. controls (235.45,172.87) and (238.73,176.15) .. (238.73,180.2) .. controls (238.73,184.25) and (235.45,187.53) .. (231.4,187.53) .. controls (227.35,187.53) and (224.07,184.25) .. (224.07,180.2) -- cycle ;
%Shape: Circle [id:dp47760906812193527] 
\draw  [color={rgb, 255:red, 155; green, 155; blue, 155 }  ,draw opacity=1 ][fill={rgb, 255:red, 231; green, 231; blue, 231 }  ,fill opacity=1 ] (254.07,180.2) .. controls (254.07,176.15) and (257.35,172.87) .. (261.4,172.87) .. controls (265.45,172.87) and (268.73,176.15) .. (268.73,180.2) .. controls (268.73,184.25) and (265.45,187.53) .. (261.4,187.53) .. controls (257.35,187.53) and (254.07,184.25) .. (254.07,180.2) -- cycle ;
%Shape: Circle [id:dp9346300500705962] 
\draw  [color={rgb, 255:red, 155; green, 155; blue, 155 }  ,draw opacity=1 ][fill={rgb, 255:red, 231; green, 231; blue, 231 }  ,fill opacity=1 ] (284.07,180.2) .. controls (284.07,176.15) and (287.35,172.87) .. (291.4,172.87) .. controls (295.45,172.87) and (298.73,176.15) .. (298.73,180.2) .. controls (298.73,184.25) and (295.45,187.53) .. (291.4,187.53) .. controls (287.35,187.53) and (284.07,184.25) .. (284.07,180.2) -- cycle ;
%Shape: Circle [id:dp9842861818223646] 
\draw  [color={rgb, 255:red, 155; green, 155; blue, 155 }  ,draw opacity=1 ][fill={rgb, 255:red, 231; green, 231; blue, 231 }  ,fill opacity=1 ] (314.07,180.2) .. controls (314.07,176.15) and (317.35,172.87) .. (321.4,172.87) .. controls (325.45,172.87) and (328.73,176.15) .. (328.73,180.2) .. controls (328.73,184.25) and (325.45,187.53) .. (321.4,187.53) .. controls (317.35,187.53) and (314.07,184.25) .. (314.07,180.2) -- cycle ;
%Shape: Circle [id:dp40520051516427336] 
\draw  [color={rgb, 255:red, 155; green, 155; blue, 155 }  ,draw opacity=1 ][fill={rgb, 255:red, 231; green, 231; blue, 231 }  ,fill opacity=1 ] (344.07,180.2) .. controls (344.07,176.15) and (347.35,172.87) .. (351.4,172.87) .. controls (355.45,172.87) and (358.73,176.15) .. (358.73,180.2) .. controls (358.73,184.25) and (355.45,187.53) .. (351.4,187.53) .. controls (347.35,187.53) and (344.07,184.25) .. (344.07,180.2) -- cycle ;
%Shape: Circle [id:dp15528006293370944] 
\draw  [color={rgb, 255:red, 155; green, 155; blue, 155 }  ,draw opacity=1 ][fill={rgb, 255:red, 231; green, 231; blue, 231 }  ,fill opacity=1 ] (374.07,180.2) .. controls (374.07,176.15) and (377.35,172.87) .. (381.4,172.87) .. controls (385.45,172.87) and (388.73,176.15) .. (388.73,180.2) .. controls (388.73,184.25) and (385.45,187.53) .. (381.4,187.53) .. controls (377.35,187.53) and (374.07,184.25) .. (374.07,180.2) -- cycle ;
%Shape: Circle [id:dp9699506872372539] 
\draw  [color={rgb, 255:red, 155; green, 155; blue, 155 }  ,draw opacity=1 ][fill={rgb, 255:red, 231; green, 231; blue, 231 }  ,fill opacity=1 ] (404.07,180.2) .. controls (404.07,176.15) and (407.35,172.87) .. (411.4,172.87) .. controls (415.45,172.87) and (418.73,176.15) .. (418.73,180.2) .. controls (418.73,184.25) and (415.45,187.53) .. (411.4,187.53) .. controls (407.35,187.53) and (404.07,184.25) .. (404.07,180.2) -- cycle ;
%Shape: Circle [id:dp46818299950096787] 
\draw  [color={rgb, 255:red, 155; green, 155; blue, 155 }  ,draw opacity=1 ][fill={rgb, 255:red, 231; green, 231; blue, 231 }  ,fill opacity=1 ] (434.07,180.2) .. controls (434.07,176.15) and (437.35,172.87) .. (441.4,172.87) .. controls (445.45,172.87) and (448.73,176.15) .. (448.73,180.2) .. controls (448.73,184.25) and (445.45,187.53) .. (441.4,187.53) .. controls (437.35,187.53) and (434.07,184.25) .. (434.07,180.2) -- cycle ;
%Shape: Circle [id:dp6709414886844423] 
\draw  [color={rgb, 255:red, 155; green, 155; blue, 155 }  ,draw opacity=1 ][fill={rgb, 255:red, 231; green, 231; blue, 231 }  ,fill opacity=1 ] (194.07,210.2) .. controls (194.07,206.15) and (197.35,202.87) .. (201.4,202.87) .. controls (205.45,202.87) and (208.73,206.15) .. (208.73,210.2) .. controls (208.73,214.25) and (205.45,217.53) .. (201.4,217.53) .. controls (197.35,217.53) and (194.07,214.25) .. (194.07,210.2) -- cycle ;
%Shape: Circle [id:dp6242940184204586] 
\draw  [color={rgb, 255:red, 155; green, 155; blue, 155 }  ,draw opacity=1 ][fill={rgb, 255:red, 231; green, 231; blue, 231 }  ,fill opacity=1 ] (224.07,210.2) .. controls (224.07,206.15) and (227.35,202.87) .. (231.4,202.87) .. controls (235.45,202.87) and (238.73,206.15) .. (238.73,210.2) .. controls (238.73,214.25) and (235.45,217.53) .. (231.4,217.53) .. controls (227.35,217.53) and (224.07,214.25) .. (224.07,210.2) -- cycle ;
%Shape: Circle [id:dp4467106467269726] 
\draw  [color={rgb, 255:red, 155; green, 155; blue, 155 }  ,draw opacity=1 ][fill={rgb, 255:red, 231; green, 231; blue, 231 }  ,fill opacity=1 ] (254.07,210.2) .. controls (254.07,206.15) and (257.35,202.87) .. (261.4,202.87) .. controls (265.45,202.87) and (268.73,206.15) .. (268.73,210.2) .. controls (268.73,214.25) and (265.45,217.53) .. (261.4,217.53) .. controls (257.35,217.53) and (254.07,214.25) .. (254.07,210.2) -- cycle ;
%Shape: Circle [id:dp24809046693738834] 
\draw  [color={rgb, 255:red, 155; green, 155; blue, 155 }  ,draw opacity=1 ][fill={rgb, 255:red, 231; green, 231; blue, 231 }  ,fill opacity=1 ] (284.07,210.2) .. controls (284.07,206.15) and (287.35,202.87) .. (291.4,202.87) .. controls (295.45,202.87) and (298.73,206.15) .. (298.73,210.2) .. controls (298.73,214.25) and (295.45,217.53) .. (291.4,217.53) .. controls (287.35,217.53) and (284.07,214.25) .. (284.07,210.2) -- cycle ;
%Shape: Circle [id:dp697801542242781] 
\draw  [color={rgb, 255:red, 155; green, 155; blue, 155 }  ,draw opacity=1 ][fill={rgb, 255:red, 231; green, 231; blue, 231 }  ,fill opacity=1 ] (314.07,210.2) .. controls (314.07,206.15) and (317.35,202.87) .. (321.4,202.87) .. controls (325.45,202.87) and (328.73,206.15) .. (328.73,210.2) .. controls (328.73,214.25) and (325.45,217.53) .. (321.4,217.53) .. controls (317.35,217.53) and (314.07,214.25) .. (314.07,210.2) -- cycle ;
%Shape: Circle [id:dp08972826025308156] 
\draw  [color={rgb, 255:red, 155; green, 155; blue, 155 }  ,draw opacity=1 ][fill={rgb, 255:red, 231; green, 231; blue, 231 }  ,fill opacity=1 ] (344.07,210.2) .. controls (344.07,206.15) and (347.35,202.87) .. (351.4,202.87) .. controls (355.45,202.87) and (358.73,206.15) .. (358.73,210.2) .. controls (358.73,214.25) and (355.45,217.53) .. (351.4,217.53) .. controls (347.35,217.53) and (344.07,214.25) .. (344.07,210.2) -- cycle ;
%Shape: Circle [id:dp9183831842625523] 
\draw  [color={rgb, 255:red, 155; green, 155; blue, 155 }  ,draw opacity=1 ][fill={rgb, 255:red, 231; green, 231; blue, 231 }  ,fill opacity=1 ] (374.07,210.2) .. controls (374.07,206.15) and (377.35,202.87) .. (381.4,202.87) .. controls (385.45,202.87) and (388.73,206.15) .. (388.73,210.2) .. controls (388.73,214.25) and (385.45,217.53) .. (381.4,217.53) .. controls (377.35,217.53) and (374.07,214.25) .. (374.07,210.2) -- cycle ;
%Shape: Circle [id:dp321332394203376] 
\draw  [color={rgb, 255:red, 155; green, 155; blue, 155 }  ,draw opacity=1 ][fill={rgb, 255:red, 231; green, 231; blue, 231 }  ,fill opacity=1 ] (404.07,210.2) .. controls (404.07,206.15) and (407.35,202.87) .. (411.4,202.87) .. controls (415.45,202.87) and (418.73,206.15) .. (418.73,210.2) .. controls (418.73,214.25) and (415.45,217.53) .. (411.4,217.53) .. controls (407.35,217.53) and (404.07,214.25) .. (404.07,210.2) -- cycle ;
%Shape: Circle [id:dp04370261916181273] 
\draw  [color={rgb, 255:red, 155; green, 155; blue, 155 }  ,draw opacity=1 ][fill={rgb, 255:red, 231; green, 231; blue, 231 }  ,fill opacity=1 ] (434.07,210.2) .. controls (434.07,206.15) and (437.35,202.87) .. (441.4,202.87) .. controls (445.45,202.87) and (448.73,206.15) .. (448.73,210.2) .. controls (448.73,214.25) and (445.45,217.53) .. (441.4,217.53) .. controls (437.35,217.53) and (434.07,214.25) .. (434.07,210.2) -- cycle ;
%Shape: Circle [id:dp060242624451909865] 
\draw  [color={rgb, 255:red, 155; green, 155; blue, 155 }  ,draw opacity=1 ][fill={rgb, 255:red, 231; green, 231; blue, 231 }  ,fill opacity=1 ] (194.07,240.2) .. controls (194.07,236.15) and (197.35,232.87) .. (201.4,232.87) .. controls (205.45,232.87) and (208.73,236.15) .. (208.73,240.2) .. controls (208.73,244.25) and (205.45,247.53) .. (201.4,247.53) .. controls (197.35,247.53) and (194.07,244.25) .. (194.07,240.2) -- cycle ;
%Shape: Circle [id:dp12306124666671914] 
\draw  [color={rgb, 255:red, 155; green, 155; blue, 155 }  ,draw opacity=1 ][fill={rgb, 255:red, 231; green, 231; blue, 231 }  ,fill opacity=1 ] (224.07,240.2) .. controls (224.07,236.15) and (227.35,232.87) .. (231.4,232.87) .. controls (235.45,232.87) and (238.73,236.15) .. (238.73,240.2) .. controls (238.73,244.25) and (235.45,247.53) .. (231.4,247.53) .. controls (227.35,247.53) and (224.07,244.25) .. (224.07,240.2) -- cycle ;
%Shape: Circle [id:dp03758061185734107] 
\draw  [color={rgb, 255:red, 155; green, 155; blue, 155 }  ,draw opacity=1 ][fill={rgb, 255:red, 231; green, 231; blue, 231 }  ,fill opacity=1 ] (254.07,240.2) .. controls (254.07,236.15) and (257.35,232.87) .. (261.4,232.87) .. controls (265.45,232.87) and (268.73,236.15) .. (268.73,240.2) .. controls (268.73,244.25) and (265.45,247.53) .. (261.4,247.53) .. controls (257.35,247.53) and (254.07,244.25) .. (254.07,240.2) -- cycle ;
%Shape: Circle [id:dp693594846548127] 
\draw  [color={rgb, 255:red, 155; green, 155; blue, 155 }  ,draw opacity=1 ][fill={rgb, 255:red, 231; green, 231; blue, 231 }  ,fill opacity=1 ] (284.07,240.2) .. controls (284.07,236.15) and (287.35,232.87) .. (291.4,232.87) .. controls (295.45,232.87) and (298.73,236.15) .. (298.73,240.2) .. controls (298.73,244.25) and (295.45,247.53) .. (291.4,247.53) .. controls (287.35,247.53) and (284.07,244.25) .. (284.07,240.2) -- cycle ;
%Shape: Circle [id:dp73781766692443] 
\draw  [color={rgb, 255:red, 155; green, 155; blue, 155 }  ,draw opacity=1 ][fill={rgb, 255:red, 231; green, 231; blue, 231 }  ,fill opacity=1 ] (314.07,240.2) .. controls (314.07,236.15) and (317.35,232.87) .. (321.4,232.87) .. controls (325.45,232.87) and (328.73,236.15) .. (328.73,240.2) .. controls (328.73,244.25) and (325.45,247.53) .. (321.4,247.53) .. controls (317.35,247.53) and (314.07,244.25) .. (314.07,240.2) -- cycle ;
%Shape: Circle [id:dp5668702750483593] 
\draw  [color={rgb, 255:red, 155; green, 155; blue, 155 }  ,draw opacity=1 ][fill={rgb, 255:red, 231; green, 231; blue, 231 }  ,fill opacity=1 ] (344.07,240.2) .. controls (344.07,236.15) and (347.35,232.87) .. (351.4,232.87) .. controls (355.45,232.87) and (358.73,236.15) .. (358.73,240.2) .. controls (358.73,244.25) and (355.45,247.53) .. (351.4,247.53) .. controls (347.35,247.53) and (344.07,244.25) .. (344.07,240.2) -- cycle ;
%Shape: Circle [id:dp8414427019524234] 
\draw  [color={rgb, 255:red, 155; green, 155; blue, 155 }  ,draw opacity=1 ][fill={rgb, 255:red, 231; green, 231; blue, 231 }  ,fill opacity=1 ] (374.07,240.2) .. controls (374.07,236.15) and (377.35,232.87) .. (381.4,232.87) .. controls (385.45,232.87) and (388.73,236.15) .. (388.73,240.2) .. controls (388.73,244.25) and (385.45,247.53) .. (381.4,247.53) .. controls (377.35,247.53) and (374.07,244.25) .. (374.07,240.2) -- cycle ;
%Shape: Circle [id:dp8108835843202838] 
\draw  [color={rgb, 255:red, 155; green, 155; blue, 155 }  ,draw opacity=1 ][fill={rgb, 255:red, 231; green, 231; blue, 231 }  ,fill opacity=1 ] (404.07,240.2) .. controls (404.07,236.15) and (407.35,232.87) .. (411.4,232.87) .. controls (415.45,232.87) and (418.73,236.15) .. (418.73,240.2) .. controls (418.73,244.25) and (415.45,247.53) .. (411.4,247.53) .. controls (407.35,247.53) and (404.07,244.25) .. (404.07,240.2) -- cycle ;
%Shape: Circle [id:dp682855564812166] 
\draw  [color={rgb, 255:red, 155; green, 155; blue, 155 }  ,draw opacity=1 ][fill={rgb, 255:red, 231; green, 231; blue, 231 }  ,fill opacity=1 ] (434.07,240.2) .. controls (434.07,236.15) and (437.35,232.87) .. (441.4,232.87) .. controls (445.45,232.87) and (448.73,236.15) .. (448.73,240.2) .. controls (448.73,244.25) and (445.45,247.53) .. (441.4,247.53) .. controls (437.35,247.53) and (434.07,244.25) .. (434.07,240.2) -- cycle ;
%Shape: Circle [id:dp25062490868621856] 
\draw  [color={rgb, 255:red, 155; green, 155; blue, 155 }  ,draw opacity=1 ][fill={rgb, 255:red, 231; green, 231; blue, 231 }  ,fill opacity=1 ] (192.07,271) .. controls (192.07,266.95) and (195.35,263.67) .. (199.4,263.67) .. controls (203.45,263.67) and (206.73,266.95) .. (206.73,271) .. controls (206.73,275.05) and (203.45,278.33) .. (199.4,278.33) .. controls (195.35,278.33) and (192.07,275.05) .. (192.07,271) -- cycle ;
%Shape: Circle [id:dp20246999079700545] 
\draw  [color={rgb, 255:red, 155; green, 155; blue, 155 }  ,draw opacity=1 ][fill={rgb, 255:red, 231; green, 231; blue, 231 }  ,fill opacity=1 ] (224.07,270.2) .. controls (224.07,266.15) and (227.35,262.87) .. (231.4,262.87) .. controls (235.45,262.87) and (238.73,266.15) .. (238.73,270.2) .. controls (238.73,274.25) and (235.45,277.53) .. (231.4,277.53) .. controls (227.35,277.53) and (224.07,274.25) .. (224.07,270.2) -- cycle ;
%Shape: Circle [id:dp23633316864332743] 
\draw  [color={rgb, 255:red, 155; green, 155; blue, 155 }  ,draw opacity=1 ][fill={rgb, 255:red, 231; green, 231; blue, 231 }  ,fill opacity=1 ] (254.07,270.2) .. controls (254.07,266.15) and (257.35,262.87) .. (261.4,262.87) .. controls (265.45,262.87) and (268.73,266.15) .. (268.73,270.2) .. controls (268.73,274.25) and (265.45,277.53) .. (261.4,277.53) .. controls (257.35,277.53) and (254.07,274.25) .. (254.07,270.2) -- cycle ;
%Shape: Circle [id:dp06493234264063241] 
\draw  [color={rgb, 255:red, 155; green, 155; blue, 155 }  ,draw opacity=1 ][fill={rgb, 255:red, 231; green, 231; blue, 231 }  ,fill opacity=1 ] (284.07,270.2) .. controls (284.07,266.15) and (287.35,262.87) .. (291.4,262.87) .. controls (295.45,262.87) and (298.73,266.15) .. (298.73,270.2) .. controls (298.73,274.25) and (295.45,277.53) .. (291.4,277.53) .. controls (287.35,277.53) and (284.07,274.25) .. (284.07,270.2) -- cycle ;
%Shape: Circle [id:dp34796899314894036] 
\draw  [color={rgb, 255:red, 155; green, 155; blue, 155 }  ,draw opacity=1 ][fill={rgb, 255:red, 231; green, 231; blue, 231 }  ,fill opacity=1 ] (314.07,270.2) .. controls (314.07,266.15) and (317.35,262.87) .. (321.4,262.87) .. controls (325.45,262.87) and (328.73,266.15) .. (328.73,270.2) .. controls (328.73,274.25) and (325.45,277.53) .. (321.4,277.53) .. controls (317.35,277.53) and (314.07,274.25) .. (314.07,270.2) -- cycle ;
%Shape: Circle [id:dp807785873143843] 
\draw  [color={rgb, 255:red, 155; green, 155; blue, 155 }  ,draw opacity=1 ][fill={rgb, 255:red, 231; green, 231; blue, 231 }  ,fill opacity=1 ] (344.07,270.2) .. controls (344.07,266.15) and (347.35,262.87) .. (351.4,262.87) .. controls (355.45,262.87) and (358.73,266.15) .. (358.73,270.2) .. controls (358.73,274.25) and (355.45,277.53) .. (351.4,277.53) .. controls (347.35,277.53) and (344.07,274.25) .. (344.07,270.2) -- cycle ;
%Shape: Circle [id:dp20099546247483413] 
\draw  [color={rgb, 255:red, 155; green, 155; blue, 155 }  ,draw opacity=1 ][fill={rgb, 255:red, 231; green, 231; blue, 231 }  ,fill opacity=1 ] (374.07,270.2) .. controls (374.07,266.15) and (377.35,262.87) .. (381.4,262.87) .. controls (385.45,262.87) and (388.73,266.15) .. (388.73,270.2) .. controls (388.73,274.25) and (385.45,277.53) .. (381.4,277.53) .. controls (377.35,277.53) and (374.07,274.25) .. (374.07,270.2) -- cycle ;
%Shape: Circle [id:dp3548869966465408] 
\draw  [color={rgb, 255:red, 155; green, 155; blue, 155 }  ,draw opacity=1 ][fill={rgb, 255:red, 231; green, 231; blue, 231 }  ,fill opacity=1 ] (404.07,270.2) .. controls (404.07,266.15) and (407.35,262.87) .. (411.4,262.87) .. controls (415.45,262.87) and (418.73,266.15) .. (418.73,270.2) .. controls (418.73,274.25) and (415.45,277.53) .. (411.4,277.53) .. controls (407.35,277.53) and (404.07,274.25) .. (404.07,270.2) -- cycle ;
%Shape: Circle [id:dp567690302702572] 
\draw  [color={rgb, 255:red, 155; green, 155; blue, 155 }  ,draw opacity=1 ][fill={rgb, 255:red, 231; green, 231; blue, 231 }  ,fill opacity=1 ] (434.07,270.2) .. controls (434.07,266.15) and (437.35,262.87) .. (441.4,262.87) .. controls (445.45,262.87) and (448.73,266.15) .. (448.73,270.2) .. controls (448.73,274.25) and (445.45,277.53) .. (441.4,277.53) .. controls (437.35,277.53) and (434.07,274.25) .. (434.07,270.2) -- cycle ;
%Shape: Circle [id:dp12732811593503401] 
\draw  [color={rgb, 255:red, 155; green, 155; blue, 155 }  ,draw opacity=1 ][fill={rgb, 255:red, 231; green, 231; blue, 231 }  ,fill opacity=1 ] (194.07,300.2) .. controls (194.07,296.15) and (197.35,292.87) .. (201.4,292.87) .. controls (205.45,292.87) and (208.73,296.15) .. (208.73,300.2) .. controls (208.73,304.25) and (205.45,307.53) .. (201.4,307.53) .. controls (197.35,307.53) and (194.07,304.25) .. (194.07,300.2) -- cycle ;
%Shape: Circle [id:dp06442323630812019] 
\draw  [color={rgb, 255:red, 155; green, 155; blue, 155 }  ,draw opacity=1 ][fill={rgb, 255:red, 231; green, 231; blue, 231 }  ,fill opacity=1 ] (224.07,300.2) .. controls (224.07,296.15) and (227.35,292.87) .. (231.4,292.87) .. controls (235.45,292.87) and (238.73,296.15) .. (238.73,300.2) .. controls (238.73,304.25) and (235.45,307.53) .. (231.4,307.53) .. controls (227.35,307.53) and (224.07,304.25) .. (224.07,300.2) -- cycle ;
%Shape: Circle [id:dp7119105126081211] 
\draw  [color={rgb, 255:red, 155; green, 155; blue, 155 }  ,draw opacity=1 ][fill={rgb, 255:red, 231; green, 231; blue, 231 }  ,fill opacity=1 ] (254.07,300.2) .. controls (254.07,296.15) and (257.35,292.87) .. (261.4,292.87) .. controls (265.45,292.87) and (268.73,296.15) .. (268.73,300.2) .. controls (268.73,304.25) and (265.45,307.53) .. (261.4,307.53) .. controls (257.35,307.53) and (254.07,304.25) .. (254.07,300.2) -- cycle ;
%Shape: Circle [id:dp20581141722989704] 
\draw  [color={rgb, 255:red, 155; green, 155; blue, 155 }  ,draw opacity=1 ][fill={rgb, 255:red, 231; green, 231; blue, 231 }  ,fill opacity=1 ] (284.07,300.2) .. controls (284.07,296.15) and (287.35,292.87) .. (291.4,292.87) .. controls (295.45,292.87) and (298.73,296.15) .. (298.73,300.2) .. controls (298.73,304.25) and (295.45,307.53) .. (291.4,307.53) .. controls (287.35,307.53) and (284.07,304.25) .. (284.07,300.2) -- cycle ;
%Shape: Circle [id:dp969237151294416] 
\draw  [color={rgb, 255:red, 155; green, 155; blue, 155 }  ,draw opacity=1 ][fill={rgb, 255:red, 231; green, 231; blue, 231 }  ,fill opacity=1 ] (314.07,300.2) .. controls (314.07,296.15) and (317.35,292.87) .. (321.4,292.87) .. controls (325.45,292.87) and (328.73,296.15) .. (328.73,300.2) .. controls (328.73,304.25) and (325.45,307.53) .. (321.4,307.53) .. controls (317.35,307.53) and (314.07,304.25) .. (314.07,300.2) -- cycle ;
%Shape: Circle [id:dp760313876343911] 
\draw  [color={rgb, 255:red, 155; green, 155; blue, 155 }  ,draw opacity=1 ][fill={rgb, 255:red, 231; green, 231; blue, 231 }  ,fill opacity=1 ] (344.07,300.2) .. controls (344.07,296.15) and (347.35,292.87) .. (351.4,292.87) .. controls (355.45,292.87) and (358.73,296.15) .. (358.73,300.2) .. controls (358.73,304.25) and (355.45,307.53) .. (351.4,307.53) .. controls (347.35,307.53) and (344.07,304.25) .. (344.07,300.2) -- cycle ;
%Shape: Circle [id:dp9581177608426413] 
\draw  [color={rgb, 255:red, 155; green, 155; blue, 155 }  ,draw opacity=1 ][fill={rgb, 255:red, 231; green, 231; blue, 231 }  ,fill opacity=1 ] (374.07,300.2) .. controls (374.07,296.15) and (377.35,292.87) .. (381.4,292.87) .. controls (385.45,292.87) and (388.73,296.15) .. (388.73,300.2) .. controls (388.73,304.25) and (385.45,307.53) .. (381.4,307.53) .. controls (377.35,307.53) and (374.07,304.25) .. (374.07,300.2) -- cycle ;
%Shape: Circle [id:dp21180451770740527] 
\draw  [color={rgb, 255:red, 155; green, 155; blue, 155 }  ,draw opacity=1 ][fill={rgb, 255:red, 231; green, 231; blue, 231 }  ,fill opacity=1 ] (404.07,300.2) .. controls (404.07,296.15) and (407.35,292.87) .. (411.4,292.87) .. controls (415.45,292.87) and (418.73,296.15) .. (418.73,300.2) .. controls (418.73,304.25) and (415.45,307.53) .. (411.4,307.53) .. controls (407.35,307.53) and (404.07,304.25) .. (404.07,300.2) -- cycle ;
%Shape: Circle [id:dp43546901896558876] 
\draw  [color={rgb, 255:red, 155; green, 155; blue, 155 }  ,draw opacity=1 ][fill={rgb, 255:red, 231; green, 231; blue, 231 }  ,fill opacity=1 ] (434.07,300.2) .. controls (434.07,296.15) and (437.35,292.87) .. (441.4,292.87) .. controls (445.45,292.87) and (448.73,296.15) .. (448.73,300.2) .. controls (448.73,304.25) and (445.45,307.53) .. (441.4,307.53) .. controls (437.35,307.53) and (434.07,304.25) .. (434.07,300.2) -- cycle ;
%Curve Lines [id:da7819431068959589] 
\draw [color={rgb, 255:red, 208; green, 2; blue, 27 }  ,draw opacity=1 ][line width=1.5]    (261.4,127.53) .. controls (261.92,122.08) and (262.92,119.58) .. (268.73,120.2) ;
%Curve Lines [id:da14518232347112014] 
\draw [color={rgb, 255:red, 208; green, 2; blue, 27 }  ,draw opacity=1 ][line width=1.5]    (261.4,157.53) .. controls (261.92,152.08) and (263.92,148.58) .. (261.4,142.87) ;
%Curve Lines [id:da2149990137201755] 
\draw [color={rgb, 255:red, 208; green, 2; blue, 27 }  ,draw opacity=1 ][line width=1.5]    (381.4,187.53) .. controls (381.92,182.08) and (383.92,178.58) .. (381.4,172.87) ;
%Curve Lines [id:da8311173202249925] 
\draw [color={rgb, 255:red, 208; green, 2; blue, 27 }  ,draw opacity=1 ][line width=1.5]    (261.4,217.53) .. controls (261.92,212.08) and (259.42,208.58) .. (261.4,202.87) ;
%Curve Lines [id:da888624265616709] 
\draw [color={rgb, 255:red, 208; green, 2; blue, 27 }  ,draw opacity=1 ][line width=1.5]    (261.4,187.53) .. controls (264.42,180.58) and (258.67,180.08) .. (261.4,172.87) ;
%Curve Lines [id:da13864723159095116] 
\draw [color={rgb, 255:red, 208; green, 2; blue, 27 }  ,draw opacity=1 ][line width=1.5]    (381.4,157.53) .. controls (384.42,150.58) and (378.67,150.08) .. (381.4,142.87) ;
%Curve Lines [id:da7183553712285776] 
\draw [color={rgb, 255:red, 208; green, 2; blue, 27 }  ,draw opacity=1 ][line width=1.5]    (381.4,217.53) .. controls (384.42,210.58) and (378.67,210.08) .. (381.4,202.87) ;
%Curve Lines [id:da895095056122264] 
\draw [color={rgb, 255:red, 208; green, 2; blue, 27 }  ,draw opacity=1 ][line width=1.5]    (358.73,240.2) .. controls (351.17,236.83) and (350.17,242.08) .. (344.07,240.2) ;
%Curve Lines [id:da39517284574035705] 
\draw [color={rgb, 255:red, 208; green, 2; blue, 27 }  ,draw opacity=1 ][line width=1.5]    (328.73,240.2) .. controls (321.92,242.08) and (320.17,242.08) .. (314.07,240.2) ;
%Curve Lines [id:da26361668803997673] 
\draw [color={rgb, 255:red, 208; green, 2; blue, 27 }  ,draw opacity=1 ][line width=1.5]    (298.73,240.2) .. controls (292.67,242.08) and (290.67,237.58) .. (284.07,240.2) ;
%Curve Lines [id:da30134516662482624] 
\draw [color={rgb, 255:red, 208; green, 2; blue, 27 }  ,draw opacity=1 ][line width=1.5]    (298.73,120.2) .. controls (292.67,122.08) and (290.67,117.58) .. (284.07,120.2) ;
%Curve Lines [id:da055185362686639094] 
\draw [color={rgb, 255:red, 208; green, 2; blue, 27 }  ,draw opacity=1 ][line width=1.5]    (328.73,120.2) .. controls (322.67,122.08) and (320.17,121.58) .. (314.07,120.2) ;
%Curve Lines [id:da26943159893852] 
\draw [color={rgb, 255:red, 208; green, 2; blue, 27 }  ,draw opacity=1 ][line width=1.5]    (358.73,120.2) .. controls (353.42,117.83) and (350.17,121.58) .. (344.07,120.2) ;
%Curve Lines [id:da5520214877895043] 
\draw [color={rgb, 255:red, 208; green, 2; blue, 27 }  ,draw opacity=1 ][line width=1.5]    (381.4,127.53) .. controls (384.42,119.58) and (380.17,121.58) .. (374.07,120.2) ;
%Curve Lines [id:da870079504747281] 
\draw [color={rgb, 255:red, 208; green, 2; blue, 27 }  ,draw opacity=1 ][line width=1.5]    (374.07,240.2) .. controls (381.92,242.58) and (381.92,238.08) .. (381.4,232.87) ;
%Curve Lines [id:da8935444861229214] 
\draw [color={rgb, 255:red, 208; green, 2; blue, 27 }  ,draw opacity=1 ][line width=1.5]    (268.73,240.2) .. controls (255.42,243.08) and (261.92,238.08) .. (261.4,232.87) ;
%Straight Lines [id:da3929450608652061] 
\draw [color={rgb, 255:red, 74; green, 144; blue, 226 }  ,draw opacity=1 ][fill={rgb, 255:red, 231; green, 231; blue, 231 }  ,fill opacity=1 ][line width=1.5]    (268.73,120.2) ;
\draw [shift={(268.73,120.2)}, rotate = 0] [color={rgb, 255:red, 74; green, 144; blue, 226 }  ,draw opacity=1 ][fill={rgb, 255:red, 74; green, 144; blue, 226 }  ,fill opacity=1 ][line width=1.5]      (0, 0) circle [x radius= 1.74, y radius= 1.74]   ;
%Straight Lines [id:da40825322091771743] 
\draw [color={rgb, 255:red, 74; green, 144; blue, 226 }  ,draw opacity=1 ][fill={rgb, 255:red, 231; green, 231; blue, 231 }  ,fill opacity=1 ][line width=1.5]    (284.07,120.2) ;
\draw [shift={(284.07,120.2)}, rotate = 0] [color={rgb, 255:red, 74; green, 144; blue, 226 }  ,draw opacity=1 ][fill={rgb, 255:red, 74; green, 144; blue, 226 }  ,fill opacity=1 ][line width=1.5]      (0, 0) circle [x radius= 1.74, y radius= 1.74]   ;
%Straight Lines [id:da024910703439943394] 
\draw [color={rgb, 255:red, 74; green, 144; blue, 226 }  ,draw opacity=1 ][fill={rgb, 255:red, 231; green, 231; blue, 231 }  ,fill opacity=1 ][line width=1.5]    (298.73,120.2) ;
\draw [shift={(298.73,120.2)}, rotate = 0] [color={rgb, 255:red, 74; green, 144; blue, 226 }  ,draw opacity=1 ][fill={rgb, 255:red, 74; green, 144; blue, 226 }  ,fill opacity=1 ][line width=1.5]      (0, 0) circle [x radius= 1.74, y radius= 1.74]   ;
%Straight Lines [id:da717924901746671] 
\draw [color={rgb, 255:red, 74; green, 144; blue, 226 }  ,draw opacity=1 ][fill={rgb, 255:red, 231; green, 231; blue, 231 }  ,fill opacity=1 ][line width=1.5]    (314.07,120.2) ;
\draw [shift={(314.07,120.2)}, rotate = 0] [color={rgb, 255:red, 74; green, 144; blue, 226 }  ,draw opacity=1 ][fill={rgb, 255:red, 74; green, 144; blue, 226 }  ,fill opacity=1 ][line width=1.5]      (0, 0) circle [x radius= 1.74, y radius= 1.74]   ;
%Straight Lines [id:da7488304664878976] 
\draw [color={rgb, 255:red, 74; green, 144; blue, 226 }  ,draw opacity=1 ][fill={rgb, 255:red, 231; green, 231; blue, 231 }  ,fill opacity=1 ][line width=1.5]    (328.73,120.2) ;
\draw [shift={(328.73,120.2)}, rotate = 0] [color={rgb, 255:red, 74; green, 144; blue, 226 }  ,draw opacity=1 ][fill={rgb, 255:red, 74; green, 144; blue, 226 }  ,fill opacity=1 ][line width=1.5]      (0, 0) circle [x radius= 1.74, y radius= 1.74]   ;
%Straight Lines [id:da6947904447790099] 
\draw [color={rgb, 255:red, 74; green, 144; blue, 226 }  ,draw opacity=1 ][fill={rgb, 255:red, 231; green, 231; blue, 231 }  ,fill opacity=1 ][line width=1.5]    (344.07,120.2) ;
\draw [shift={(344.07,120.2)}, rotate = 0] [color={rgb, 255:red, 74; green, 144; blue, 226 }  ,draw opacity=1 ][fill={rgb, 255:red, 74; green, 144; blue, 226 }  ,fill opacity=1 ][line width=1.5]      (0, 0) circle [x radius= 1.74, y radius= 1.74]   ;
%Straight Lines [id:da1036902959684678] 
\draw [color={rgb, 255:red, 74; green, 144; blue, 226 }  ,draw opacity=1 ][fill={rgb, 255:red, 231; green, 231; blue, 231 }  ,fill opacity=1 ][line width=1.5]    (358.73,120.2) ;
\draw [shift={(358.73,120.2)}, rotate = 0] [color={rgb, 255:red, 74; green, 144; blue, 226 }  ,draw opacity=1 ][fill={rgb, 255:red, 74; green, 144; blue, 226 }  ,fill opacity=1 ][line width=1.5]      (0, 0) circle [x radius= 1.74, y radius= 1.74]   ;
%Straight Lines [id:da9974067585884957] 
\draw [color={rgb, 255:red, 74; green, 144; blue, 226 }  ,draw opacity=1 ][fill={rgb, 255:red, 231; green, 231; blue, 231 }  ,fill opacity=1 ][line width=1.5]    (374.07,120.2) ;
\draw [shift={(374.07,120.2)}, rotate = 0] [color={rgb, 255:red, 74; green, 144; blue, 226 }  ,draw opacity=1 ][fill={rgb, 255:red, 74; green, 144; blue, 226 }  ,fill opacity=1 ][line width=1.5]      (0, 0) circle [x radius= 1.74, y radius= 1.74]   ;
%Shape: Rectangle [id:dp8378007196214312] 
\draw  [draw opacity=0][fill={rgb, 255:red, 255; green, 255; blue, 255 }  ,fill opacity=1 ] (165.06,320.29) -- (477.74,320.29) -- (477.74,340.11) -- (165.06,340.11) -- cycle ;
%Shape: Rectangle [id:dp5855622430887087] 
\draw  [draw opacity=0][fill={rgb, 255:red, 255; green, 255; blue, 255 }  ,fill opacity=1 ] (461.01,23.45) -- (481.79,23.45) -- (481.79,330.2) -- (461.01,330.2) -- cycle ;

% Text Node
\draw (388.2,218.6) node [anchor=north west][inner sep=0.75pt]  [color={rgb, 255:red, 208; green, 2; blue, 27 }  ,opacity=1 ]  {$C$};
% Text Node
\draw (360.05,99.2) node [anchor=north west][inner sep=0.75pt]  [color={rgb, 255:red, 74; green, 144; blue, 226 }  ,opacity=1 ]  {$F$};

\end{tikzpicture}

%% file: references.bib
@article{benjamini2012separation,
  title={On the separation profile of infinite graphs},
  author={Benjamini, Itai and Schramm, Oded and Tim{\'a}r, {\'A}d{\'a}m},
  journal={Groups, Geometry, and Dynamics},
  volume={6},
  number={4},
  pages={639--658},
  year={2012}
}

@article{hume2026separation,
  title={Separation profiles of free products},
  author={Hume, David},
  journal={arXiv preprint arXiv:2604.24462},
  year={2026}
}

@article{grigoriev2011tree,
  title={Tree-width and large grid minors in planar graphs},
  author={Grigoriev, Alexander},
  journal={Discrete Mathematics \& Theoretical Computer Science},
  volume={13},
  number={Graph and Algorithms},
  year={2011},
  publisher={Episciences. org}
}

@InProceedings{kisfaludi2023separator,
  author =	{Kisfaludi-Bak, S\'{a}ndor and Masa\v{r}{\'\i}kov\'{a}, Jana and van Leeuwen, Erik Jan and Walczak, Bartosz and W\k{e}grzycki, Karol},
  title =	{{Separator theorem and algorithms for planar hyperbolic graphs}},
  booktitle =	{40th International Symposium on Computational Geometry (SoCG 2024)},
  pages =	{67:1--67:17},
  series =	{Leibniz International Proceedings in Informatics (LIPIcs)},
  ISBN =	{978-3-95977-316-4},
  ISSN =	{1868-8969},
  year =	{2024},
  volume =	{293},
  editor =	{Mulzer, Wolfgang and Phillips, Jeff M.},
  publisher =	{Schloss Dagstuhl -- Leibniz-Zentrum f{\"u}r Informatik},
  address =	{Dagstuhl, Germany},
  URL =		{https://drops.dagstuhl.de/entities/document/10.4230/LIPIcs.SoCG.2024.67},
  URN =		{urn:nbn:de:0030-drops-200126},
  doi =		{10.4230/LIPIcs.SoCG.2024.67}
}

@article{berger2024bounded,
  title={Bounded-diameter tree-decompositions},
  author={Berger, Eli and Seymour, Paul},
  journal={Combinatorica},
  volume={44},
  number={3},
  pages={659--674},
  year={2024},
  publisher={Springer}
}

@article{dvovrak2019treewidth,
  title={Treewidth of graphs with balanced separations},
  author={Dvo{\v{r}}{\'a}k, Zden{\v{e}}k and Norin, Sergey},
  journal={Journal of Combinatorial Theory, Series B},
  volume={137},
  pages={137--144},
  year={2019},
  publisher={Elsevier}
}

@incollection{gromov1987hyperbolic,
  title={Hyperbolic groups},
  author={Gromov, Mikhael},
  booktitle={Essays in group theory},
  pages={75--263},
  year={1987},
  publisher={Springer}
}

@article{chuzhoy2021towards,
  title={Towards tight(er) bounds for the excluded grid theorem},
  author={Chuzhoy, Julia and Tan, Zihan},
  journal={Journal of Combinatorial Theory, Series B},
  volume={146},
  pages={219--265},
  year={2021},
  publisher={Elsevier}
}

@article{lipton1979separator,
  title={A separator theorem for planar graphs},
  author={Lipton, Richard J and Tarjan, Robert Endre},
  journal={SIAM Journal on Applied Mathematics},
  volume={36},
  number={2},
  pages={177--189},
  year={1979},
  publisher={SIAM}
}

@book{ghys2013groupes,
  title={Sur les groupes hyperboliques d’apr{\`e}s {Mikhael Gromov}},
  author={Ghys, Etienne and de la Harpe, Pierre},
  volume={83},
  year={2013},
  publisher={Springer Science \& Business Media}
}

@article{seymour1994call,
  title={Call routing and the ratcatcher},
  author={Seymour, Paul D. and Thomas, Robin},
  journal={Combinatorica},
  volume={14},
  number={2},
  pages={217--241},
  year={1994},
  publisher={Springer}
}

@inproceedings{chepoi2008diameters,
  title={Diameters, centers, and approximating trees of delta-hyperbolic geodesic spaces and graphs},
  author={Chepoi, Victor and Dragan, Feodor F. and Estellon, Bertrand and Habib, Michel and Vax{\`e}s, Yann},
  booktitle={Proceedings of the twenty-fourth annual symposium on Computational geometry},
  pages={59--68},
  year={2008}
}

@article{dieng2009tree,
  title={On the tree-width of planar graphs},
  author={Dieng, Youssou and Gavoille, Cyril},
  journal={Electronic Notes in Discrete Mathematics},
  volume={34},
  pages={593--596},
  year={2009},
  publisher={Elsevier}
}

@article{ChepoiDraganEstellonHabibVaxesXiang2012,
  author  = {Chepoi, Victor and Dragan, Feodor F. and Estellon, Bertrand
             and Habib, Michel and Vax{\`e}s, Yann and Xiang, Yang},
  title   = {Additive Spanners and Distance and Routing Labeling Schemes
             for Hyperbolic Graphs},
  journal = {Algorithmica},
  volume  = {62},
  number  = {3--4},
  pages   = {713--732},
  year    = {2012},
  doi     = {10.1007/s00453-010-9478-x}
}

@article{CoudertDucoffeNisse2016,
  author  = {Coudert, David and Ducoffe, Guillaume and Nisse, Nicolas},
  title   = {To Approximate Treewidth, Use Treelength!},
  journal = {SIAM Journal on Discrete Mathematics},
  volume  = {30},
  number  = {3},
  pages   = {1424--1436},
  year    = {2016},
  doi     = {10.1137/15M1034039}
}

@book{BridsonHaefliger1999,
  author    = {Bridson, Martin R. and Haefliger, Andr{\'e}},
  title     = {Metric Spaces of Non-Positive Curvature},
  series    = {Grundlehren der mathematischen Wissenschaften},
  volume    = {319},
  publisher = {Springer},
  year      = {1999}
}

@article{BielakPanczyk2012,
  author  = {Bielak, Halina and Pa{\'n}czyk, Micha{\l}},
  title   = {A self-stabilizing algorithm for finding weighted centroid in trees},
  journal = {Annales UMCS Informatica},
  volume  = {12},
  number  = {2},
  pages   = {27--37},
  year    = {2012},
  doi     = {10.2478/v10065-012-0035-x}
}

@article{houdrouge2025separation,
  title={Separation number and treewidth, revisited},
  author={Houdrouge, Hussein and Miraftab, Babak and Morin, Pat},
  journal={arXiv preprint arXiv:2503.17112},
  year={2025}
}

@article{hume2022poincare,
  title={Poincar{\'e} profiles of Lie groups and a coarse geometric dichotomy},
  author={Hume, David and Mackay, John M and Tessera, Romain},
  journal={Geometric and Functional Analysis},
  volume={32},
  number={5},
  pages={1063--1133},
  year={2022},
  publisher={Birkhaeuser Science}
}

@article{hume2020poincare,
  title={Poincar{\'e} profiles of groups and spaces},
  author={Hume, David and Mackay, John M and Tessera, Romain},
  journal={Revista Matem{\'a}tica Iberoamericana},
  volume={36},
  number={6},
  pages={1835--1886},
  year={2020}
}

@inproceedings{gournay2023separation,
  title={Separation profile, isoperimetry, growth and compression},
  author={Gournay, Antoine and Le Coz, Corentin},
  booktitle={Annales de l'Institut Fourier},
  volume={73},
  number={4},
  pages={1627--1675},
  year={2023}
}

@article{berger2005glauber,
  title={Glauber dynamics on trees and hyperbolic graphs},
  author={Berger, Noam and Kenyon, Claire and Mossel, Elchanan and Peres, Yuval},
  journal={Probability Theory and Related Fields},
  volume={131},
  number={3},
  pages={311--340},
  year={2005},
  publisher={Springer}
}
